\documentclass[a4paper,12pt]{amsart}

\usepackage{amsmath, amsthm, amssymb}
\usepackage[dvipsnames]{xcolor}
\usepackage{hyperref}
\usepackage{comment}
\usepackage{setspace}
\usepackage[backgroundcolor=white,bordercolor=white]{todonotes}
\usepackage{enumerate}
\usepackage{fancyhdr}
\usepackage[normalem]{ulem}

\newtheorem{thm}{Theorem}[section]
\newtheorem{cor}[thm]{Corollary}
\newtheorem{prop}[thm]{Proposition}
\newtheorem{lem}[thm]{Lemma}

\theoremstyle{remark}
\newtheorem{rem}[thm]{Remark}

\theoremstyle{definition}
\newtheorem{defi}[thm]{Definition}

\theoremstyle{theorem}
\newtheorem{asmp}[thm]{Assumption}

\newcommand{\R}{\ensuremath{\mathbb{R}}}
\newcommand{\calS}{\ensuremath{\mathcal{S}}}

\DeclareMathOperator*{\esssup}{ess \, sup}

\usepackage{color}

\definecolor{gr}{rgb}   {0.,   0.8,   0. } 
\definecolor{bl}{rgb}   {0.,   0.5,   1. } 
\definecolor{mg}{rgb}   {0.7,  0.,    0.7}

\title{On the Well-posedness of  Magnetic Schr\"odinger Equations with Unbounded Potentials}
\author{Dorothee Frey}
\address{Department of Mathematics, Karlsruhe Institute of Technology, 76131 Karlsruhe, Germany}
\email{dorothee.frey@kit.edu}

\author{Siliang Weng}
\address{Department of Mathematics, Karlsruhe Institute of Technology, 76131 Karlsruhe, Germany}
\email{siliang.weng@kit.edu}

\thanks{Funded by the Deutsche Forschungsgemeinschaft (DFG, German Research Foundation) – Project-ID 258734477 – SFB 1173. The second author wishes to thank H. Cornean for very helpful discussions.}



\begin{document}
	
	\pagestyle{plain}

	%\doublespacing
	%\onehalfspacing
	
	%\tableofcontents
	
	\begin{abstract}
		We consider  magnetic Schr\"odinger equations with  sublinear  magnetic potentials and  subquadratic  electric potentials   on $\mathbb{R}^{d}$, as well as generalizations thereof.  We obtain  new results on the   global well-posedness  of the Cauchy problem with initial data in  magnetic modulation spaces  $M^{p}_{A}(\mathbb{R}^{d})$.   Our  results are achieved by  approximating the solution in phase space  using the magnetic Hamiltonian flow.  This method includes the potentials as part of the generalized Schr\"odinger operator instead of treating them as perturbations,  and  thereby allows us to deal with unbounded potentials.   For $A \equiv 0$, the space $M^{p}_{A}(\mathbb{R}^{d})$ reduces to the usual modulation space $M^{p}(\mathbb{R}^{d})$, for which relevant known results for the usual Schr\"odinger equation can be recovered.

	\end{abstract}
	
	%\date{\today}
	
	\maketitle
	
	\section{Introduction}

	Let $d \geq 2$. In this article, we study the global well-posedness of the Cauchy problem for the  Schr\"odinger equation  with unbounded electromagnetic potentials $A,V$  in $\mathbb{R}\times\mathbb{R}^{d}$,
	\begin{equation}\label{eq:usual magnetic Schrodinger}
		i\partial_{t}u(t;y) = (H^{A} + V(t))u(t;y),\qquad u(0) = u_{0}
	\end{equation}
 on magnetic modulation spaces $M^{p}_{A}(\mathbb{R}^{d})$ as defined below.  
Here $H^{A}$ denotes the magnetic Laplacian with the magnetic vector potential $A = (A_{1},A_{2},\dots,A_{d})$, that is
	$$
	H^{A} := \frac{1}{2} \sum_{j=1}^d\left(-i \partial_j-A_j(y)\right)^{2},
	$$
	and $V(t) \in C^{\infty}(\mathbb{R}^{d},\mathbb{R})$ denotes the time-dependent electric scalar potential. We consider unbounded potentials $A$ with at most linear growth and $V(t)$ with at most quadratic growth, that is for multi-indices $\alpha$ and $\beta$, the potentials satisfy
	$$
	|\partial^{\alpha}A(y)| \leq C_{A,\alpha}, \qquad  |\partial^{\beta}_{y}V(t;y)| \leq C_{V,\beta}, \qquad   |\alpha| \geq 1,\ |\beta| \geq 2.
	$$
	We also define the magnetic field $B$ to be a smooth, closed $2$-form $$B(y):=\sum_{1 \leq j, k \leq d} B_{j k}(y) d y_j \wedge d y_k$$ with the skew symmetry $B_{jk} = -B_{kj} \in C^{\infty}(\mathbb{R}^{d},\mathbb{R})$. Then a magnetic vector potential $A$ can be associated to $B$ if it satisfies $dA = B$, that is $\partial_{k}A_{j}-\partial_{j}A_{k} = B_{jk}$.	
	
	For  differential operators with long-range magnetic potentials, a pseudo-differential calculus  adapted to the magnetic potential   has been developed  by Iftimie, Măntoiu, and Purice in \cite{Mantoiu_Weyl}, \cite{Iftimie_MagneticPsiDO}. The theory extends the usual Weyl calculus to the magnetic Weyl calculus $h \mapsto \operatorname{Op}^{A}(h)$, mapping a symbol $h \in \mathcal{S'}(\mathbb{R}^{2d})$ to a magnetic pseudo-differential operator (magnetic $\Psi$DO) $\operatorname{Op}^{A}(h)$, acting on $u \in \mathcal{S}(\mathbb{R}^{d})$  as  
	\begin{equation}\label{def:magnetic psido}
		\operatorname{Op}^{A}(h)u(y) := (2\pi)^{-d}\int_{\mathbb{R}^{d}\times\mathbb{R}^{d}} e^{i(y-y')\cdot \eta} e^{-i\varphi^{A}(y',y)} h(\tfrac{y+y'}{2},\eta)u(y')\ dy'd\eta.
	\end{equation}
	The precise meaning of the above expression will be recalled in the next section.  We note here that this adapted formulation differs from the usual Weyl calculus only by a magnetic-dependent phase term $e^{-i\varphi^{A}(y',y)}$, which has the property
	\begin{equation}\label{eq:magnetic phase 0}
		e^{-i\varphi^{A}(y',y)}\big|_{A=0} \equiv 1.
	\end{equation}
We thus recover the usual Weyl calculus  as the special case  $A\equiv 0$. With this formulation, the magnetic Schr\"odinger operator in \eqref{eq:usual magnetic Schrodinger}  can be expressed as  
	\begin{equation}\label{eq:usual magnetic Schrodinger symbol}
		H^{A}+V = \operatorname{Op}^{A}(h),\qquad  h(t;y,\eta) = |\eta|^{2} + V(t;y).
	\end{equation}
	We show under suitable conditions, the Cauchy problems associated to a family of equations that contains \eqref{eq:usual magnetic Schrodinger}
	\begin{equation}\label{eq:general magnetic Schrodinger}
		i\partial_{t}u(t,y) = \operatorname{Op}^{A}(h)u(t;y),\quad u(0) = u_{0}
	\end{equation}
	is well-posed  on   the magnetic modulation space $M^{p}_{A}(\mathbb{R}^{d})$, for  $1\leq p \leq \infty$. Here the space $M^{p}_{A}(\mathbb{R}^{d})$ can be seen as the $A$-adapted version of the usual modulation space, defined through the magnetic wavepacket transform $\mathcal{T}^{A}$ for $u \in \mathcal{S'}(\mathbb{R}^{d})$
	$$
	\mathcal{T}^{A}u(x,\xi) := \int e^{-i\xi\cdot(y-x)}e^{-i\varphi^{A}(y,x)}g(y-x)u(y)\ dy,\quad  \|u\|_{M^{p}_{A}(\mathbb{R}^{d})} := \left\|\mathcal{T}^{A}u\right\|_{L^{p}(\mathbb{R}^{2d})},
	$$
	where $g$ denotes the normalized Gaussian function on $\mathbb{R}^{d}$ (a window function). In particular, for $p=2$ we recover the usual $L^{2}$ space, that is $M^{2}_{A}(\mathbb{R}^{d}) \cong L^{2}(\mathbb{R}^{d})$. We also note, although the space $M^{p}_{A}$ is formulated explicitly with the magnetic potential $A$, it  in fact intrinsically depends  only  on the magnetic field $B$. In other words, $M^{p}_{A}$ is gauge-independent, in the sense that for another magnetic potential $A'$ with $dA' = dA = B$, we have $M^{p}_{A} \simeq M^{p}_{A'}$ isometrically.
	
	This choice of space is inspired by the well-posedness of the usual Schr\"odinger equation in modulation spaces $M^{p,q}$, see \cite{Benyi_Unimodular}, \cite{Benyi_Modulation} and the references therein. Also due to the phase term $e^{-i\varphi^{A}}$ and its property \eqref{eq:magnetic phase 0}, the magnetic modulation space $M^{p}_{A}$ naturally extends the usual modulation space $M^{p,p}$, and is well-behaved interacting with the adapted magnetic $\Psi$DOs \eqref{def:magnetic psido}.  One of the early ideas of defining magnetic versions of modulation spaces can be found in Măntoiu-Purice \cite{Mantoiu_Modulation}, where they considered the case  $p=2$. To the best of our knowledge, the definition and discussion of $M^{p}_{A}$ and its weighted version $M^{m,p}_{A}$ in this article is new, and so are its applications to magnetic Schr\"odinger equations.
	
	Previously well-posedness results for the Cauchy problem of \eqref{eq:usual magnetic Schrodinger} with unbounded potentials are mostly obtained on $L^{2}$-based spaces. One foundational work is Yajima \cite{Yajima_Schrodinger}, where unbounded, non-smooth potentials with $t$-dependence are considered. More generally, Doi in \cite{Doi_Smoothness} showed under certain conditions the variable-coefficient version of \eqref{eq:usual magnetic Schrodinger} with unbounded potentials is also well-posed in $L^{2}$-based weighted Sobolev spaces. On the other hand, results for $L^{p}$-based spaces with $p \neq 2$ seem out of reach, since the non-magnetic Schr\"odinger group $e^{-it\Delta}$ is already unbounded on $L^{p}$ for $p \neq 2$. For the non-magnetic $A\equiv 0$ case, the introduction of the modulation space $M^{p,q}$ as a substitute for $L^{p}$ has been very successful, since the Schr\"odinger group $e^{-it\Delta}$ has been shown to be bounded on the modulation spaces $M^{p,q}$ \cite{Benyi_Unimodular}. This suggests taking spaces of modulation type as candidates for wellposedness of magnetic Schr\"odinger equations. In this direction, Kato-Muramatsu \cite{Kato_Sub-linear} obtained a-priori estimates on $M^{p}$, but there the existence of the solution is still restricted to $L^{2}$.
	
	By introducing the magnetic modulation space $M^{p}_{A}$  and the weighted version $M^{m,p}_{A}$, we obtain a global well-posedness result for \eqref{eq:general magnetic Schrodinger}.  Our assumptions on the potentials are the following. 
	
	\begin{asmp}\label{asmp:general}
		Fix $\epsilon > 0$. We assume the magnetic field $B$ has  smooth components $\{B_{jk}\}_{1\leq j,k\leq d}$, which are all bounded with $\|B_{jk}\|_{\infty} \leq C_{B}$ for some $C_B>0$. Furthermore, we assume the derivatives of $B_{jk}$ satisfy a decay condition, that is  for multi-indices $\alpha$, there exists $C_{\alpha,B}>0$ such that 
		\begin{equation}\label{eq:B decay}
			|\partial^{\alpha}B_{jk}(y)| \leq C_{\alpha,B}  \langle y\rangle^{-1-\epsilon}, \qquad  |\alpha| \geq 1.
		\end{equation}
		For such given magnetic field $B$, we define the magnetic vector potential $A$ by the transversal gauge
		\begin{equation}\label{def:A}
			A_j(y):=-\sum_{k=1}^d \int_0^1 sy_k B_{j k}(sy)\  d s, \quad y \in \mathbb{R}^{d}.
		\end{equation}
		For the electric scalar potential $V$, we allow $t$-dependence and assume that $V \in C(\mathbb{R},C^{\infty}(\mathbb{R}))$, with quadratic growth bounds that is uniform in $t$,
		\begin{equation}\label{def:V}
			|\partial^{\alpha}_{y}V(t;y)| \leq C_{V,\alpha},\qquad  |\alpha| \geq 2, \; t \in \mathbb{R}.
		\end{equation}
	\end{asmp}
	
	For the magnetic $\Psi$DO $\operatorname{Op}^{A}(h)$ generalizing the magnetic Schr\"odinger operator, we define the symbol class $S^{0,(k)}$, which aligns with the symbol class used in Tataru \cite{Tataru_Phase}.
	\begin{defi}
		For $k \in \mathbb{N}$ and multi-indices $\alpha,\beta$, a symbol $h \in C^{\infty}(\mathbb{R}^{2d})$ is in the symbol class $S^{0,(k)}(\mathbb{R}^{2d})$ if it satisfies
		$$
			|\partial_{y}^{\alpha}\partial_{\eta}^{\beta}h(y,\eta)| \leq C_{\alpha,\beta},\qquad |\alpha|+|\beta| \geq k.
		$$
	\end{defi}
	
	We make the following assumption on the symbols.
	\begin{asmp}\label{asmp:h}
		We assume the time-dependent symbol $h \in C^{1}(\mathbb{R},C^{\infty}(\mathbb{R}^{2d}))$ is real-valued, and $h(t), \dot{h}(t) \in S^{0,(2)}(\mathbb{R}^{2d})$ uniformly in $t$.  Furthermore, we assume $h$ is of second-order in $\eta$, that is, for all $t\in \mathbb{R}$ there exists  $C>0$ such that $|h(t;0,\eta)|\geq C|\eta|^{2}$.
	\end{asmp}
	
	\begin{rem}
	 While the second-order part of the assumption guarantees the operator $\operatorname{Op}^{A}(h)$ to be of  magnetic Schr\"odinger operator type, our result also holds free of this part of the assumption. In fact, for $h$ that is first-order in $\eta$, the arguments will be simpler. 
	\end{rem}
	
	The above assumption gives a family of symbols which contains the symbol $h(t;y,\eta) = |\eta|^{2}+V(t;y)$ in \eqref{eq:usual magnetic Schrodinger symbol} with $V$ satisfying  Assumption \ref{asmp:general}. Our main result is the following. 
	
	\begin{thm}\label{thm:main} 
	Let $T > 0$ and $1\leq p \leq \infty$. For a given magnetic field $B$ satisfying Assumption \ref{asmp:general} and a real-valued symbol $h$ satisfying Assumption \ref{asmp:h}, the Cauchy problem with initial value  $u_{0} \in M^{2,p}_{A}(\mathbb{R}^{d})$ 
		\begin{equation}\label{eq:electro-magnetic weyl schrodinger}
			\begin{cases}
				\ i\partial_{t}u(t,y) = \operatorname{Op}^{A}(h)u(t,y),\\
				\ u(0,y) = u_{0}(y),
			\end{cases} 
		\end{equation}
		has a unique strong solution $u \in W^{1,1}([0,T],M_{A}^{p}(\mathbb{R}^{d}))$ with $u(t) \in M^{2,p}_{A}(\mathbb{R}^{d})$ for all $t \in [0,T]$. 
		Moreover, we have the bound
		$$
		\|u(t)\|_{M^{2,p}_{A}(\mathbb{R}^{d})} \leq C_{T,h,B}\|u_{0}\|_{M^{2,p}_{A}(\mathbb{R}^{d})}.
		$$
	\end{thm}
	This result gives the following simple corollary.
	\begin{cor}
		Let $T > 0$ and $1\leq p\leq \infty, s \geq 2$. For the same magnetic Schr\"odinger operator $\operatorname{Op}^{A}(h)$ as in Theorem \ref{thm:main} and  $f \in L^{1}([0,T],M^{s,p}_{A}(\mathbb{R}^{d}))$, the Cauchy problem with initial value $u_{0} \in M^{2,p}_{A}(\mathbb{R}^{d})$ 
		$$
			\begin{cases}
				\ i\partial_{t}u(t,y) = \operatorname{Op}^{A}(h)u(t,y) + f(t,y),\\
				\ u(0,y) = u_{0}(y)
			\end{cases} 
		$$
		has a unique strong solution $u \in W^{1,1}([0,T],M^{p}_{A}(\mathbb{R}^{d}))$ with $u(t) \in M^{2,p}_{A}(\mathbb{R}^{d})$ for all $t \in [0,T]$.
	\end{cor}
	
	Our approach is also flexible enough to treat the case of time-dependent magnetic potentials $A(t)$, which we explain in Section 8.
	
	\begin{thm}\label{thm:time-dependent}
		 Let $T > 0$ and $1\leq p \leq \infty$. Let $B(t)$ be a  time-dependent magnetic field with components $B_{jk} \in C^{1}(\mathbb{R},C^{\infty}(\mathbb{R}^{d}))$. For some fixed $\epsilon > 0$, assume for all $t \in \mathbb{R}$, $B_{jk}(t)$  satisfies $\|B(t)\|_{\infty} + \|\dot{B}(t)\|_{\infty} \leq C_{B}$ and the decay condition 
		\begin{equation}\label{eq:B(t) bounds}
			|\partial^{\alpha}B_{jk}(t;y)| + |\partial^{\alpha}\dot{B}_{jk}(t;y)| \leq C_{\alpha,B}\langle y\rangle^{-1-\epsilon},\qquad  |\alpha| \geq 1.
		\end{equation}
		The Cauchy problem with initial value  $u_{0} \in M^{2,p}_{A(0)}(\mathbb{R}^{d})$ 
		\begin{equation}\label{eq:electro-magnetic weyl schrodinger with A(t)}
			\begin{cases}
				\ i\partial_{t}u(t,y) = \operatorname{Op}^{A(t)}(h)u(t,y),\\
				\ u(0,y) = u_{0}(y),
			\end{cases} 
		\end{equation}
		has a unique  strong solution $u \in W^{1,1}([0,T],M_{A(t)}^{p}(\mathbb{R}^{d}))$ with $u(t) \in M^{2,p}_{A(t)}(\mathbb{R}^{d})$ for all $t \in [0,T]$. Moreover, we have the bound
		$$
		\|u(t)\|_{M^{2,p}_{A(t)}(\mathbb{R}^{d})} \leq C_{T,h,B}\|u_{0}\|_{M^{2,p}_{A(0)}(\mathbb{R}^{d})}.
		$$
	\end{thm}
	With the identification $M^{2}_{A(t)} \simeq L^{2}$, the above result can be seen as a generalization of the classical results in Yajima \cite{Yajima_Schrodinger}. 
	
	Our strategy for the proof is inspired by the work of Tataru \cite{Tataru_Phase}, Smith \cite{Smith_Spectral}, Cordero-Gr\"ochenig-Nicola-Rodino \cite{Cordero_Wiener}, Cordero-Nicola-Rodino \cite{Cordero_Rogh} and Cornean-Helffer-Purice \cite{Cornean_Matrix}. The main idea of the strategy is to take a phase space perspective by using the wavepacket transform $\mathcal{T}^{A}$, so that the quantum evolution $u_{0} \mapsto u(t)$ satisfying (\ref{eq:electro-magnetic weyl schrodinger}) can be related to a classical evolution in the phase space, that is a magnetic Hamiltonian flow with initial data $\tilde{u}_{0} := \mathcal{T}^{A}u_{0}$. The flow property is then exploited through a generalized version of Yajima  \cite[Lemma 2.1]{Yajima_Schrodinger}, which only requires the mild decay assumption on $B$ as in \eqref{eq:B decay}. The $A$-adapted space $M^{p}_{A}$ naturally arises from the use of the transform $\mathcal{T}^{A}$, and for $A \equiv 0$ our result recovers the usual modulation space well-posedness results for the non-magnetic Schr\"odinger  equation  with unbounded electric potential $V$.

	\section{Preliminaries}
	
	\subsection{Notation} For the configuration space $\mathbb{R}^{d}$, the phase space is the cotangent bundle $T^{*}\mathbb{R}^{d} \simeq \mathbb{R}^{2d}$. For differential operators, we use the notation $D := -i\partial$, and for a scalar function $f$ on $\mathbb{R}^{d}$, we use $\partial_{k}f(x) := \partial_{x_{k}}f(x)$ where $x := (x_{1},\dots,x_{k},\dots,x_{d})$. If $f$ is time-dependent, we also use the shorthand notation $\dot{f} := df/dt$.  For scalar products of $x,y \in \mathbb{R}^{d}$ we use $x\cdot y$ or $\langle x,y\rangle$. For the $L^{2}$ pairing we denote $\langle f,g\rangle_{L^{2}} := \int \bar{f}g$, with the complex conjugation on the left. 
	
	We use  the Japanese bracket $\langle x\rangle := (1+|x|^{2})^{1/2}$ for $x \in \R^d$, and note  Peetre's inequality, for $s \in \mathbb{R}$,
	$$
	\langle x + y \rangle^{s} \leq 2^{s}\langle x\rangle^{s}\langle y\rangle^{|s|}, 
	\qquad x,y \in \R^d. 
	$$
	For inequalities, we use the notation $\lesssim$ to hide absolute constants, and also the constants that only depend on invariant objects such as the Gaussian $g$. We use subscripts to denote the dependence of relevant objects or parameters, for example $C_{B}$ denotes constants that can be determined from a given magnetic field $B$.
	
	\subsection{Magnetic Potential and Flux}
	
	Let the magnetic field $B$ with components $B_{jk} \in C^{\infty}(\mathbb{R}^{d},\mathbb{R})$  satisfy $\|\partial^{\alpha}B_{jk}\|_\infty \leq C_{\alpha}$,  and let $A(\cdot)$ denote the magnetic vector potential  obtained through the transversal gauge (\ref{def:A}). For fixed $x \in \mathbb{R}^{d}$, we define the vector potential $A(\cdot,x)$  by 
	$$
		A_{j}(y,x) := -\sum_{k=1}^{d}\int_{0}^{1} s(y_{k}-x_{k}) B_{jk}(x+s(y-x))\ ds,\quad y \in \mathbb{R}^{d}.
	$$
	The above potential satisfies $dA(\cdot,x) = B$ and also recovers $A(y)$ by $A(y,0) = A(y)$. The vector-valued function $A(\cdot)-A(\cdot,x)$ is related to the desired gauge-covariant property of the magnetic $\Psi$DOs. Notice we have $dA(\cdot)-dA(\cdot,x)=0$, so $A(\cdot)-A(\cdot,x)$ is in fact a closed 1-form. We can thus find a function $\varphi^{A}$ such that
	\begin{equation}\label{eq:magnetic phase property}
		\partial_{y_{j}}\varphi^{A}(y,x) = A_{j}(y) - A_{j}(y,x), \quad y,x \in \mathbb{R}^{d},\ 1 \leq j \leq d.
	\end{equation}
	A straightforward choice is to integrate both sides above from $x$ to $y$ (cf. \cite{Cornean_Beal} equation (1.4)),
	\begin{equation}\label{def:magnetic phase}
		\varphi^{A}(y, x):=\int_{\Gamma_{x, y}}(A(\cdot)-A(\cdot, x))=\int_{\Gamma_{x, y}} A, \quad x,y \in \mathbb{R}^{d},
	\end{equation}
	where $\Gamma_{x, y}$ denotes the line segment connecting $x$ to $y$. Explicitly, we have
	$$
	\int_{\Gamma_{x, y}} A := \int_{\Gamma_{x, y}} A(\cdot, 0) = (y-x)\cdot \int_{0}^{1} A((1-s)x+sy,0)\ ds,
	$$
	and the $\int_{\Gamma_{x, y}}A(\cdot, x)$ term in (\ref{def:magnetic phase}) vanishes, because $B$ is skew-symmetric, so
	$$
	\int_{\Gamma_{x, y}}A(\cdot, x) = \sum_{j, k=1}^{d} (y_{j}-x_{j}) \int_0^1  \int_{0}^{1} \theta s(y_{k}-x_{k})B_{jk}(x+\theta((1-s)x+sy))\ d\theta ds = 0.
	$$
	From this construction, we see that $\varphi^{A}$ is anti-symmetric in its variables, i.e.
	\begin{equation}\label{eq:phase anti-symmetry}
		\varphi^{A}(y,x) = -\varphi^{A}(x,y),\quad x,y \in \mathbb{R}^{d}.
	\end{equation}
	Because the value of $\varphi^{A}$ equals integrating the vector field $A$ along a line segment, we denote  by  $\Gamma^{B}(y,x,z)$ the flux given by the magnetic field $B$ through the oriented triangle with vertices $y,x,z$.  Using  Stokes' theorem, we have 
	\begin{equation}\label{def:magnetic flux}
		\Gamma^{B}(y,x,z) = \varphi^{A}(y,x)+\varphi^{A}(x,z)+\varphi^{A}(z,y).
	\end{equation}
	 By \cite[Lemma 2.10]{Iftimie_MagneticFIO}, we can formulate a matrix $C(y,x,z) := \left(c_{jk}(y,x,z)\right)_{1\leq j,k\leq d}$ with
	$$
	c_{jk}(y,x,z) := \int_{0}^{1}\int_{0}^{s} B_{jk}(ty+(s-t)x+(1-s)z)\ dtds,\quad (y,x,z) \in \mathbb{R}^{3d},
	$$
	it then allows us to re-write $\Gamma^{B}$ as 
	$$
	\Gamma^{B}(y,x,z) = \langle C(y,x,z)(y-x),(y-z)\rangle.
	$$
	With the above formulation, the following estimates hold by direct calculation, which bounds the derivatives of the magnetic flux by the length of the line segments $|y-x|$ and $|y-z|$. 
	
 \begin{lem}
		For a magnetic field $B$ with $B_{jk} \in C^{\infty}(\mathbb{R}^{d},\mathbb{R})$ and  $\|\partial^{\alpha}B_{jk}\|_\infty \leq C_{\alpha}$, the associated magnetic flux $\Gamma^{B}$ and phase function $e^{i\Gamma^{B}}$ satisfy
		\begin{align}\label{eq:flux bound} \nonumber 
				|\Gamma^{B}(y,x,z)| & \leq C_{B}|y-x||y-z|,\\
				|\partial^{\alpha} e^{i\Gamma^{B}(y,x,z)}| & \leq C_{\alpha,B}(|y-x|+|y-z|)^{|\alpha|}, \qquad |\alpha| \geq 1.
		\end{align}
	\end{lem}

	\section{Magnetic Pseudo-differential Operator}
	
	In this section, we recall the construction of magnetic pseudo-differential operators ($\Psi$DOs) with a given magnetic field $B$  satisfying  $\sup_{x \in \mathbb{R}^{d}}|\partial^{\alpha}B_{jk}(x)| \leq C_{\alpha}$. With  $\varphi^{A}$ constructed in \eqref{def:magnetic phase}, we can now give a definition of magnetic $\Psi$DOs  following the construction in \cite[Section 1]{Iftimie_MagneticPsiDO}),   for general symbols $h$ in the space of tempered distributions $\mathcal{S'}(\mathbb{R}^{2d})$. Since tempered distributions can grow at most polynomially at infinity, the resulting magnetic $\Psi$DOs as integral operators are well-defined for Schwartz functions $\mathcal{S}(\mathbb{R}^{d})$.

	\begin{defi}
		Given a symbol $h \in \mathcal{S'}(\mathbb{R}^{2d})$, the magnetic pseudo-differential operator $\operatorname{Op}^{A}(h) \in \mathcal{L}(\mathcal{S}(\mathbb{R}^{d});\mathcal{S'}(\mathbb{R}^{d}))$ is defined by
		\begin{equation}\label{def:quantization}
			\operatorname{Op}^{A}(h)u(y) := (2\pi)^{-d}\int_{\mathbb{R}^{2d}} e^{i(y-y')\cdot \eta} e^{-i\varphi^{A}(y',y)} h(\tfrac{y+y'}{2},\eta)u(y')\ dy'd\eta, \quad y \in \mathbb{R}^{d},
		\end{equation}
		where   $u \in \calS(\R^d)$. 
	\end{defi}
	
	This mapping of symbols to magnetic $\Psi$DOs $h \mapsto \operatorname{Op}^{A}(h)$ is also called magnetic Weyl quantization. This formulation has the advantage that the magnetic potential $A$ is fully decoupled from the symbol $h$. For comparison, the magnetic Schr\"odinger operator $H^{A}+V$ in (\ref{eq:usual magnetic Schrodinger}) can also be formulated with the usual $\Psi$DO theory, which requires a much complicated symbol 
	$$
	h_{A}(t;y,\eta) := |\eta-A(y)|^{2} + V(t;y)
	$$
	such that
	\begin{equation*}\label{eq:Weyl calculus}
		H^{A} + V = \operatorname{Op}(h_{A}) := (2\pi)^{-d}\int_{\mathbb{R}^{2d}} e^{i(y-y')\cdot \eta} h_{A}(\tfrac{y+y'}{2},\eta)u(y')\ dy'd\eta.
	\end{equation*}
	Furthermore, the magnetic $\Psi$DO formulation enjoys a gauge-covariant property. That is, for $A' = A + dv$ with $v \in C_{\mathrm{pol}}^{\infty}\left(\mathbb{R}^d,\mathbb{R}\right)$ a real-valued smooth function of at most polynomial growth, $A$ and $A'$ are both associated to the magnetic field $B = dA'= dA$. Then the magnetic $\Psi$DOs $ \operatorname{Op}^{A}(h)$ and $\operatorname{Op}^{A'}(h)$ are  unitarily equivalent in $L^{2}$ (cf. \cite[Proposition 3.6]{Mantoiu_Weyl}). Thus this formulation depends more intrinsically on the object $B$, which allows us to formulate Assumption \ref{asmp:general} in terms of the magnetic field $B$ instead of  the magnetic potentials $A$. In particular, the adapted magnetic modulation space $M^{p}_{A}(\mathbb{R}^{d})$ we formulate also depends intrinsically on $B$, which we show in the next section as Proposition \ref{prop:magnetic modulation space}. 
	
	Another important feature of the above  Weyl  quantization is that for real symbols $h$, the operator $\operatorname{Op}^{A}(h)$ is symmetric on $\mathcal{S}(\mathbb{R}^{d})$, thanks to the balanced choice $\frac{y+y'}{2}$ in (\ref{def:quantization}). But this can also make explicit calculations a bit awkward. To simplify calculations, a slightly different quantization is also useful, that is the magnetic Kohn-Nirenberg quantization $\operatorname{Op}_{KN}^{A}$ (or the magnetic ``left" quantization)
	$$
		\operatorname{Op}_{KN}^{A}(h)u(y) := (2\pi)^{-d}\int_{\mathbb{R}^{2d}} e^{i(y-y')\cdot \eta}e^{-i\varphi^{A}(y',y)} h(y,\eta)u(y')\ dy' d \eta, \quad y \in \mathbb{R}^{d}.
	$$
	The Weyl quantization can be converted to the Kohn-Nierenberg quantization by modifying the symbol $h$.  \cite[Proposition 2.13]{Iftimie_MagneticFIO}  provides such a conversion for a symbol class $S^{m}$, and the proof there can be used with obvious changes to give a conversion for our symbol class $S^{0,(2)}(\mathbb{R}^{2d})$.
	\begin{lem}\label{lem:Weyl to KN}
		Given  $a \in S^{0,(2)}(\mathbb{R}^{2d})$, there exists a symbol $b \in S^{0,(2)}(\mathbb{R}^{2d})$ such that $\operatorname{Op}^{A}(a)=\operatorname{Op}_{KN}^{A}(b)$, with
		$$
		b(y, \eta) = a(y,\eta) + a_{r}(y,\eta),\quad a_{r} \in S^{0,(0)}(\mathbb{R}^{2d}).
		$$
\end{lem}

By  \cite[Proposition 2.13]{Iftimie_MagneticFIO}, we also have an explicit expression for the symbol $a_{r}$  given by  
	$$
	a_{r}(y,\eta):=\sum_{|\alpha| \geq 1} \frac{1}{|\alpha|!}\left(\frac{1}{2}\right)^{|\alpha|}\left(D_y^\alpha \partial_\eta^\alpha a\right)(y, \eta),  \qquad (y,\eta) \in \R^{2d}.  
	$$
This implies that if the symbol $a$ depends only on $y$ or $\eta$, then the two quantizations $\operatorname{Op}^{A}(a)=\operatorname{Op}_{KN}^{A}(a)$ coincide, because  all  cross-derivatives vanish. In particular, the symbols of the basic magnetic differentiation and multiplication operator
	$$
	\operatorname{P}^{A}_{j}u(y) :=(-i\partial_{y_{j}}-A_{j}(y))u(y),\quad \operatorname{Q}_{j}u(y) := y_{j}u(y)
	$$
	coincide in  the  two quantizations. That is, for $u \in \mathcal{S}(\mathbb{R}^{d})$ we have
	\begin{equation}\label{eq:basic quantization}
		\begin{aligned}
			& \operatorname{P}^{A}_{j}u(y)=\operatorname{Op}^{A}(\eta_{j})u(y)=\operatorname{Op}_{KN}^{A}(\eta_{j})u(y);\\
			& \operatorname{Q}_{j}u(y)=\operatorname{Op}^{A}(y_{j})u(y)=\operatorname{Op}_{KN}^{A}(y_{j})u(y).
		\end{aligned}
	\end{equation}
	This can be shown with direct calculations using Fourier inversion.  We showcase the calculation with $\operatorname{Op}^{A}(\eta_{j})u(y)$, and the other identity is similar. Denote  by  $\mathcal{F}_{y'}$ and $\mathcal{F}^{-1}_{\eta}$ the Fourier transform and inversion regarding the variables $y'$ and $\eta$. We have
	$$
	\begin{aligned}
		\operatorname{Op}^{A}(\eta_{j})u(y) & =\operatorname{Op}_{KN}^{A}(\eta_{j})u(y)\\
		& = (2\pi)^{-d}\int_{\mathbb{R}^{2d}} e^{i(y-y')\cdot \eta} e^{-i\varphi^{A}(y',y)} \eta_{j} u(y')\ dy'd\eta\\
		& = \mathcal{F}_{\eta}^{-1}\mathcal{F}_{y'}\left(-i\partial_{y'_{j}} (e^{-i\varphi^{A}(\cdot,y)} u)\right)(y)\\
		& = (-i\partial_{y_{j}}-A_{j}(y))u(y),
	\end{aligned}
	$$
	where the last equality is based on (\ref{eq:magnetic phase property}) and $\varphi^{A}(y,y) = 0$, that is
	$$
	-i\partial_{y'_{j}}e^{-i\varphi^{A}(y',y)}\big|_{y'=y} = -(A_{j}(y')-A_{j}(y',y))e^{-i\varphi^{A}(y',y)}\big|_{y'=y} = A_{j}(y).
	$$
	
	\section{The Magnetic Wavepacket Transform and Modulation Spaces}
	
	For a given magnetic field $B$ and the associated magnetic potential $A$, we define a family of magnetic wavepackets  using the magnetic phase function $\varphi^{A}$ constructed in (\ref{def:magnetic phase}).
	\begin{defi}
We denote the $L^{2}$-normalized Gaussian by  
$$g(y) := (2\pi)^{-\frac{d}{2}}\pi^{-\frac{d}{4}} e^{-\frac{1}{2}y^{2}}, \qquad y \in \mathbb{R}^{d}.$$ For fixed $\lambda \geq 1$, the family of $\lambda$-rescaled magnetic wavepackets $\{g^{A,\lambda}_{x,\xi}\}_{(x,\xi) \in \mathbb{R}^{2d}}$ are Schwartz functions parameterized by $(x,\xi) \in \mathbb{R}^{2d}$,  defined by  
		\begin{equation}\label{eq:magnetic wavepacket rescaled}
			g^{A,\lambda}_{x,\xi}(y) := \lambda^{\frac{d}{2}}e^{i\xi\cdot (y-x)}g(\lambda(y-x))e^{i\varphi^{A}(y,x)}, \quad y \in \mathbb{R}^{d}.
		\end{equation}
	\end{defi}

By the above definition, for $A \equiv 0$ and $\lambda = 1$, we recover the standard Gaussian wavepackets $\{g_{x,\xi}\}_{(x,\xi) \in \mathbb{R}^{2d}}$,
	$$
	g_{x,\xi}(y) := e^{i\xi\cdot(y-x)}g(y-x) = \pi^{-\frac{d}{4}} e^{i\xi\cdot(y-x)} e^{-\frac{1}{2}(y-x)^{2}},\quad y \in \mathbb{R}^{d},
	$$
	which are translated and modulated Gaussians, centered around $x$ in the configuration space and around $\xi$ in the frequency space. Thus, the magnetic wavepacket transform we define below can be seen as a generalization of the usual wavepacket transform (cf. \cite{Tataru_Phase}, equation (1)).

	\begin{defi}
		For a tempered distribution $u \in \mathcal{S}'(\mathbb{R}^{d})$, its magnetic wavepacket transform $\mathcal{T}_{\lambda}^{A}u \in \mathcal{S}'(\mathbb{R}^{2d})$ is obtained by integrating against wavepackets $g^{A,\lambda}_{x,\xi}$ defined in (\ref{eq:magnetic wavepacket rescaled}), i.e., 
		$$
			\begin{aligned}
				\mathcal{T}_{\lambda}^{A}u(x,\xi) & := \int_{\mathbb{R}^{d}} \overline{g^{A,\lambda}_{x,\xi}}(y)u(y)\ dy\\
				& = \int_{\mathbb{R}^{d}} \lambda^{\frac{d}{2}}e^{-i\xi\cdot (y-x)}g(\lambda(y-x))e^{-i\varphi^{A}(y,x)} u(y)\ dy, \quad (x,\xi) \in \mathbb{R}^{2d}.
			\end{aligned}
		$$
	\end{defi}
	
	This magnetic wavepacket transform $\mathcal{T}_{\lambda}^{A}$ is an example of  a  phase space transform, since it takes a function (distribution) on the configuration space $\mathbb{R}^{d}$ to the phase space $T^{*}\mathbb{R}^{d} \simeq \mathbb{R}^{2d}$. We also define a formal adjoint $\mathcal{T}_{\lambda}^{A*}$ acting on phase space distributions $\tilde{u} \in \mathcal{S}'(\mathbb{R}^{2d})$ by 
	$$
		\mathcal{T}_{\lambda}^{A*}\tilde{u}(y) := \int_{\mathbb{R}^{2d}} g^{A,\lambda}_{x,\xi}(y) \tilde{u}(x,\xi)\ dx d\xi.
	$$
	
	\begin{lem}\label{lem:wavepacket transform}
		The magnetic wavepacket transform $\mathcal{T}_{\lambda}^{A}$ and its formal adjoint $\mathcal{T}_{\lambda}^{A*}$ have the following properties:
\begin{enumerate}[(i)]
			\item If $ u \in \mathcal{S}(\mathbb{R}^d)$, then $\mathcal{T}^{A}_{\lambda}u \in \mathcal{S}(\mathbb{R}^{2 d})$. If $\tilde{u} \in \mathcal{S}(\mathbb{R}^{2d})$ then $\mathcal{T}_{\lambda}^{A*}\tilde{u} \in \mathcal{S}(\mathbb{R}^{d})$.
			\item $\mathcal{T}_{\lambda}^{A}, \mathcal{T}_{\lambda}^{A*}$ are $L^{2}$ isometries.
			\item For $u \in \mathcal{S}(\mathbb{R}^{d})$, we have the inversion formula  $u = \mathcal{T}_{\lambda}^{A*}\mathcal{T}_{\lambda}^{A}u$. 
\end{enumerate}
	\end{lem}
	The first two properties can be checked by direct calculation, or proved similarly as in the non-magnetic case, such as Proposition 1.82 in \cite{Benyi_Modulation} (also cf. Proposition 3.3 in \cite{Mantoiu_Modulation}). The inversion formula can be obtained by
	$$
		\begin{aligned}
			\mathcal{T}_{\lambda}^{A*}\mathcal{T}_{\lambda}^{A}u(y) & = \lambda^{d} \int_{\mathbb{R}^{3d}} e^{i\xi\cdot (y-y')}g(\lambda(y-x))g(\lambda(y'-x))e^{i\varphi^{A}(y,x)}e^{-i\varphi^{A}(y',x)} u(y')\ dy'dxd\xi\\
			& = \lambda^{d}(2\pi)^{\frac{d}{2}}\int_{\mathbb{R}^{2d}} \delta(y-y') g(\lambda(y-x))g(\lambda(y'-x))e^{i\Gamma^{B}(y,x,y')}e^{-i\varphi^{A}(y',y)} u(y')\ dy'dx\\
			& = \lambda^{d}(2\pi)^{\frac{d}{2}} u(y) \int_{\mathbb{R}^{d}} g(\lambda(y-x))^{2}\ dx\\
			& = u(y),
		\end{aligned}
	$$
	where we used the evaluations
	$$
	\Gamma^{B}(y,x,y')\big|_{y'=y} = \varphi^{A}(y,x) + \varphi^{A}(x,y)  = 0,\quad e^{-i\varphi^{A}(y',y)}\big|_{y'=y} = 0.
	$$
	From the above calculation we can also see the nice interaction between two wavepackets.

	We then define the magnetic modulation space $M_{A}^{p}(\mathbb{R}^{d})$ adapted to the magnetic potential $A$ to be the space of functions whose $\mathcal{T}^{A}$ transform has  $L^p_{x,\xi}$  integrability.  
	
	\begin{defi}
		For $1 \leq p < \infty$, we define the magnetic modulation space $M_{A}^{p}(\mathbb{R}^{d})$ by
		$$
		M_{A}^{p}(\mathbb{R}^{d}) := \left\{u \in \mathcal{S}'(\mathbb{R}^{d}) : \|u\|_{M_{A}^{p}(\mathbb{R}^{d})} := \left(\int_{\mathbb{R}^{2d}}|\mathcal{T}_{\lambda}^{A}u(x,\xi)|^{p}\ dxd\xi\right)^{1/p} < \infty \right\}.
		$$
		For $p=\infty$, the integral should be replaced by $L^{\infty}$ norm and the space $M_{A}^{p}(\mathbb{R}^{d})$ is defined as the closure of the Schwartz space with respect to the norm.
	\end{defi}
	
	By this definition, when $A \equiv 0$ we recover the usual modulation space $M^{p,p}(\mathbb{R}^{d})$. We note it is also possible to generalize by replacing $L^p_{x,\xi}$ with $L^p_{x}L^{q}_{\xi}$ integrability, which then align with the non-magnetic modulation space $M^{p,q}(\mathbb{R}^{d})$, but for this article we only focus on the case $p=q$. General information and properties of the non-magnetic $M^{p,q}(\mathbb{R}^{d})$ spaces can be found in \cite{Grochenig_Foundations} or Chapter 2 in \cite{Benyi_Modulation}. Similar to the non-magnetic case, We can prove analogously the following properties for $M_{A}^{p}(\mathbb{R}^{d})$.
	
	\begin{prop}\label{prop:magnetic modulation space}
		 Let $1\leq p \leq \infty$. We have the following.   
		\begin{enumerate}[(i)]
			\item $M^{p}_{A}(\mathbb{R}^{d})$ is independent of the choice $\lambda \geq 1$.
			\item $M^{p}_{A}(\mathbb{R}^{d})$ is a Banach space.
			\item $\mathcal{S}(\mathbb{R}^{d})$ is dense in $M^{p}_{A}(\mathbb{R}^{d})$.
			\item For $u \in M^{p}_{A}(\mathbb{R}^{d})$, we have $u = \mathcal{T}_{\lambda}^{A*}\mathcal{T}_{\lambda}^{A}u$. 
		\end{enumerate}
		In particular, for $p = 2$ we have $M^{2}_{A}(\mathbb{R}^{d}) \simeq L^{2}(\mathbb{R}^{d})$ isometrically.
	\end{prop}
	
	Up to now, we have used the magnetic vector potential $A$ from the transversal gauge (\ref{def:A}) to define the wavepacket transform and the modulation spaces. Here, we remark on the general magnetic potentials $A'(y) := A(y) + \partial_{y}v(y)$. By replacing $A$ with $A'$, we can define accordingly $\mathcal{T}^{A'}$ and $M^{p}_{A'}(\mathbb{R}^{d})$, which actually leads to the same space.
	
	\begin{prop}
		For $v \in C_{\mathrm{pol}}^{\infty}\left(\mathbb{R}^d,\mathbb{R}\right)$ and $A'(y) := A(y) + \partial_{y}v(y)$, which satisfies $dA'=dA$, we have $M^{p}_{A'}(\mathbb{R}^{d}) \simeq M^{p}_{A}(\mathbb{R}^{d})$ isometrically.
	\end{prop}
	
	\begin{proof}
		Notice that $A'(y) := A(y) + \partial_{y}v(y)$ gives
		$$
		\varphi^{A'}(y,x) = \int_{\Gamma_{x, y}} A' = \int_{\Gamma_{x, y}} (A + \partial v) = \varphi^{A}(y,x) + v(y) - v(x).
		$$
		Consider the injective map  $u \mapsto e^{iv}u$  for $u \in M^{p}_{A}(\mathbb{R}^{d})$. Then for  $w:= e^{iv} u$, we have 
		$$
		\mathcal{T}^{A'}w(x,\xi) = \int e^{-i \xi \cdot(y-x)} g(y-x) e^{-i \varphi^{A'}(y, x)} w(y)\ d y = e^{iv(x)}\mathcal{T}^{A}u(x,\xi),
		$$
		so $w \in M^{p}_{A'}(\mathbb{R}^{d})$ with $\|w\|_{M^{p}_{A'}(\mathbb{R}^{d})} = \|u\|_{M^{p}_{A}(\mathbb{R}^{d})}$.  The map $u \mapsto e^{iv}u$ is also surjective, since  for any $w' \in M^{p}_{A'}(\mathbb{R}^{d})$, we can take $u' := e^{-iv}w'$, and by a similar calculation as above we have $u' \in M^{p}_{A}(\mathbb{R}^{d})$. 
	\end{proof}
	
	 We also introduce the weighted version $M^{m,p}_{A}(\mathbb{R}^{d})$.
	\begin{defi}
		Let $1 \leq p \leq \infty$ and $m \in \mathbb{R}$. We denote the standard weights $\nu_{m}(x,\xi) := (1+|x|^{2}+|\xi|^{2})^{m/2}$ and the re-scaled weights $\nu_{\lambda,m}(x,\xi) := (1+|\lambda^{-1}x|^{2}+|\lambda^{-1}\xi|^{2})^{m/2}$. The weighted magnetic modulation space $M_{A}^{m,p}(\mathbb{R}^{d})$ is defined by
		$$
		M_{A}^{m,p}(\mathbb{R}^{d}) := \left\{u \in \mathcal{S}'(\mathbb{R}^{d}) : \|u\|_{M_{A}^{m,p}(\mathbb{R}^{d})} := \|\nu_{\lambda,m}\mathcal{T}^{A}_{\lambda}u\|_{L^{p}(\mathbb{R}^{2d})} < \infty \right\},
		$$
		where the case $p=\infty$ is also understood as the closure of the Schwartz space with respect to the associated norm.
	\end{defi}
	
	 By the above definition, $M^{m,p}_{A}(\mathbb{R}^{d})$ is essentially a weighted $L^{p}$ space from a phase space perspective, so it also enjoys properties analogous to the ones in Proposition \ref{prop:magnetic modulation space}. For $m_{1} \geq m_{0}$, we have naturally $M^{m_{1},p}_{A}(\mathbb{R}^{d}) \subseteq M^{m_{0},p}_{A}(\mathbb{R}^{d})$.  For $p=2$, by the isomorphism $M^{0,2}_{A}(\mathbb{R}^{d}) \simeq L^{2}(\mathbb{R}^{d})$, it is easy to see for $k \in \mathbb{N}$, the space $M^{k,2}_{A}(\mathbb{R}^{d})$ contains the $L^{2}$-based weighted Sobolev space $\Sigma(k)$ used in Yajima \cite{Yajima_Schrodinger}
	$$
	\Sigma(k) = \{u \in L^{2}(\mathbb{R}^{d}) : \Sigma_{|\alpha|+|\beta|\leq k} \|x^{\alpha}\partial_{x}^{\beta}u\|^{2}_{L^{2}(\mathbb{R}^{d})} = \|u\|^{2}_{\Sigma(k)} < \infty\}.
	$$

	Moreover, for our generalized magnetic Schr\"odinger operator $\operatorname{Op}^{A}(h)$ with symbol $h$ satisfying Assumption \ref{asmp:h}, the space $M^{2,p}_{A}$ is a natural domain for $\operatorname{Op}^{A}(h)$ in $M^{p}_{A}$.

	\begin{lem}\label{lem:Op^A boundedness}
		Let $1 \leq p \leq \infty, m \in \mathbb{R}$, and let symbol $h$ satisfies the Assumption \ref{asmp:h}. Then the magnetic pseudo-differential operator $\operatorname{Op}^{A}(h): M^{m,p}_{A}(\mathbb{R}^{d}) \to M^{m-2,p}_{A}(\mathbb{R}^{d})$ is bounded.
	\end{lem}
	
	This lemma can be seen as a consequence of the following more general lemma.
	
	\begin{lem}\label{lem:weighted kernel bound}
		Let $1 \leq p \leq \infty$ and $m,s \in \mathbb{R}$. If a linear operator $\mathcal{L}: \mathcal{S}(\mathbb{R}^{d}) \to \mathcal{S}'(\mathbb{R}^{d})$ satisfies for all $N,N' \in \mathbb{N}$
		\begin{equation}\label{eq:L kernel bound}
			|\langle g^{A}_{z,\zeta}, \mathcal{L}g^{A}_{x,\xi}\rangle_{L^{2}}| \lesssim \nu_{s}(z,\zeta)\langle z-x\rangle^{-N}\langle \zeta-\xi\rangle^{-N'},
		\end{equation}
		then $\mathcal{L}: M^{m,p}_{A}(\mathbb{R}^{d}) \to M^{m-s,p}_{A}(\mathbb{R}^{d})$ is bounded.
	\end{lem}
	
	\begin{proof}
		 By  density of the Schwartz space $\mathcal{S}(\mathbb{R}^{d})$ in $M^{m,p}_{A}(\mathbb{R}^{d})$, it suffices to consider $u \in \mathcal{S}(\mathbb{R}^{d})$. Denote $G_{\mathcal{L}}(z,\zeta,x,\xi) := \langle g^{A}_{z,\zeta}, \mathcal{L}g^{A}_{x,\xi}\rangle_{L^{2}}$.
		By the reproducing formula  $u = \mathcal{T}^{A*}\mathcal{T}^{A}u$ we can write
		$$
		\begin{aligned}
			& \|\mathcal{L}u\|_{M^{m-s,p}_{A}(\mathbb{R}^{d})}^{p}\\
			= & \int_{\mathbb{R}^{2d}} \left|\nu_{m-s}(z,\zeta)\int_{\mathbb{R}^{d}} \overline{g^{A}_{z,\zeta}}(y) \mathcal{L} \left( 
			\int_{\mathbb{R}^{2d}} g^{A}_{x,\xi}(y) \mathcal{T}^{A}u(x,\xi)\ dxd\xi\right)\ dy\right|^{p}\ dzd\zeta\\
			= & \int_{\mathbb{R}^{2d}} \left|\int_{\mathbb{R}^{2d}} \nu_{m-s}(z,\zeta)G_{\mathcal{L}}(z,\zeta,x,\xi) \nu_{-m}(x,\xi)\nu_{m}(x,\xi)\mathcal{T}^{A}u(x,\xi)\ dxd\xi\right|^{p}\ dzd\zeta.
		\end{aligned}
		$$
		Notice that the inner integral on right hand side above can be seen as integrating a phase space function $\nu_{m}\mathcal{T}^{A}u$ against an integral kernel involving $G_{\mathcal{L}}$.  By the property \eqref{eq:L kernel bound}, the integral kernel can be bounded for all $N,N' \in \mathbb{N}$ by
	$$
		\begin{aligned}
			& |\nu_{m-s}(z,\zeta)G_{\mathcal{L}}(z,\zeta,x,\xi) \nu_{-m}(x,\xi)|\\
			\lesssim\ & \nu_{m}(z,\zeta) \langle z-x\rangle^{-N}\langle \zeta-\xi\rangle^{-N'} \nu_{-m}(x,\xi)\\
			\lesssim\ &  \nu_{m}(z-x,\zeta-\xi)\nu_{m}(x,\xi)\langle z-x\rangle^{-N}\langle \zeta-\xi\rangle^{-N'}\nu_{-m}(x,\xi)\\
			\lesssim\ &  \langle z-x\rangle^{-N+m}\langle \zeta-\xi\rangle^{-N'+m}.
		\end{aligned}
		$$
		With the sufficient off-diagonal decay of the kernel above, we obtain the boundedness of $\mathcal{L}$ by Young's inequality for integral operators
		$$
		\|\mathcal{L}u\|_{M^{m-s,p}_{A}(\mathbb{R}^{d})} \lesssim \|\nu_{m}\mathcal{T}^{A}u\|_{L^{p}(\mathbb{R}^{2d})} = \|u\|_{M^{m,p}_{A}(\mathbb{R}^{d})}.
		$$
	\end{proof}

	\section{Magnetic Phase Space Flow}
	
	The magnetic Weyl quantization $h \mapsto \operatorname{Op}^{A}(h)$ gives a correspondence between the classical Hamiltonian $h$ and the quantum observable $\operatorname{Op}^{A}(h)$. In the spirit of the correspondence principle, it is natural to relate the quantum evolution $u_{0} \mapsto u(t)$ on $\mathbb{R}^{d}$ to a classical phase flow on the phase space $T^{*}\mathbb{R}^{d} \simeq \mathbb{R}^{2d}$. Since our choice of quantization depends intrinsically on the magnetic field $B$, the corresponding flow should also depends on $B$ instead of the vector potential $A$. This is achieved by equipping the phase space $T^{*}\mathbb{R}^{d}$ with the twisted symplectic form $\omega_{B}$, given by
	$$
	\omega_{B(x)} := \sum_{j=1}^d d x_{j}\wedge d\xi_{j}+\frac{1}{2}\sum_{j, k} B_{j k}(x) d x_j\wedge d x_{k},
	$$
	which then gives rise to the associated Poisson bracket $\{\cdot,\cdot\}_{B}$
	$$
	\{f, g\}_B := \sum_{j=1}^d\left(\partial_{\xi_j} f \partial_{x_j} g - \partial_{x_j} f \partial_{\xi_j} g\right)+ \sum_{j, k=1}^d B_{j k}(\cdot) \partial_{\xi_j} f \partial_{\xi_k} g .
	$$
	 With the magnetic Poisson bracket above, we define the magnetic Hamiltonian flow induced by a time-dependent Hamiltonian $h$. 
	\begin{defi}
		For a given magnetic field $B$, we define $\chi_{B}(h)(t,s): (x^{s},\xi^{s}) \mapsto (x^{t},\xi^{t})$ the magnetic Hamiltonian flow induced by the Hamiltonian  $h \in C(\mathbb{R}, C^{\infty}(\mathbb{R}^{2d}))$, satisfying the  equations 
\begin{align} \label{eq:def-hamilt-flow}
		\begin{cases}
			\dot{x^{t}}_j  =  \{h(t),x^{t}_{j}\}_{B},\\
			\dot{\xi^{t}}_j  = \{h(t),\xi_{j}^{t}\}_{B}.
		\end{cases} 
\end{align}
	\end{defi}
	
	By definition  of the Poisson bracket,  the flow $(x^{t},\xi^{t})$ satisfies  more  explicitly 
	\begin{equation}\label{eq:magnetic hamiltonian flow flat}
		\dot{x^{t}_{j}} = \partial_{\xi_{j}}h(t;x^{t},\xi^{t}), \quad \dot{\xi^{t}_{j}} =  -\partial_{x_{j}}h(t;x^{t},\xi^{t}) + \sum_{k=1}^{d}B_{kj}(x^{t})\partial_{\xi_{k}}h(t;x^{t},\xi^{t}).
	\end{equation}
	 In the following, we focus on such magnetic flows associated to symbols $h$  satisfying $h(t) \in S^{0,(2)}(\mathbb{R}^{2d})$.

\begin{lem}\label{lem:flow}
 Assume $h \in C(\mathbb{R}, C^{\infty}(\mathbb{R}^{2d}))$ with  $h(t) \in S^{0,(2)}(\mathbb{R}^{2d})$ uniformly in $t$, and let  the  magnetic field $B$ satisfy  $\|B_{jk}\|_\infty \leq C_{B}$. Then the flow $\chi_{B}(h)$ exists globally in time and $\chi_{B}(h)(t,s)$ is volume-preserving for all $t, s \in \R$.
	\end{lem}
	
	\begin{proof}
		For this lemma, the $t$-dependence of $h(t)$ does not play a role, because we only need the bounds for the $x,\xi$ derivatives of $h$, which by our assumption are uniform in $t$.  We thus suppress the notation and just take $h(x,\xi) := h(t;x,\xi)$. For the existence, we firstly note  that  the vector field in (\ref{eq:magnetic hamiltonian flow flat}) associated to $h \in S^{0,(2)}(\mathbb{R}^{2d})$
		$$
		\Phi_{h} := \left(\partial_{\xi_{1}}h,\dots,\partial_{\xi_{d}}h,-\partial_{x_{1}}h+\sum_{k}B_{k1}(x)\partial_{\xi_{k}}h,\dots,-\partial_{x_{d}}h+\sum_{k}B_{kd}(x)\partial_{\xi_{k}}h\right)
		$$
		is smooth and thus locally Lipschitz, so $\chi_{B}(h)(t,s)$ exists locally in time. To extend it globally, we rule out finite time blow-up by showing, for $s,t \in \mathbb{R}$ with $s < t$,  we can take $T > 0$ such that $t-s < T$ and bound 
		\begin{equation}\label{eq:no blow-up}
			1+|x^{t}|+|\xi^{t}| \leq C_{T}(1+|x^{s}|+|\xi^{s}|).
		\end{equation}
		By the boundedness of the second-order derivatives of $h \in S^{0,(2)}(\mathbb{R}^{2d})$, we can bound its first-order derivatives by
		$$
			|\partial_{x}h(x,\xi)| + |\partial_{\xi}h(x,\xi)| \leq C_{h}(1 + |x| + |\xi|).
		$$
		Take $C_{B} \geq 1$ such that $|B(x)| \leq C_{B}$, and denote $C_{h,B} := C_{h}C_{B} \geq C_{h}$. The above bounds for the derivatives of $h$ and the expression \eqref{eq:magnetic hamiltonian flow flat} for $\dot{x^{t}}$ and $\dot{\xi^{t}}$ imply
		\begin{equation}\label{eq:t-derivatives of flow}
			|\dot{x^{t}}| \leq C_{h}(1+|x^{t}|+|\xi^{t}|),\qquad |\dot{\xi^{t}}| \leq C_{h,B}(1+|x^{t}|+|\xi^{t}|).
		\end{equation}
		We then calculate
		\begin{equation}\label{eq:t-derivative bounds}
			\begin{gathered}
				\dfrac{d|x^{t}|}{dt} = \dfrac{\dot{x^{t}}\cdot x^{t}}{|x^{t}|} \leq |\dot{x^{t}}| \leq C_{h}(1+|x^{t}|+|\xi^{t}|),\\
				\dfrac{d|\xi^{t}|}{dt} = \dfrac{\dot{\xi^{t}}\cdot \xi^{t}}{|\xi^{t}|} \leq |\dot{\xi^{t}}| \leq C_{h,B}(1+|x^{t}|+|\xi^{t}|).
			\end{gathered}
		\end{equation}
		Combining the above bounds, we see that
		$$
		\begin{aligned}
			\frac{d(1+|x^{t}|+|\xi^{t}|)}{dt} & \leq 2C_{h,B}(1 + |x^{t}|+|\xi^{t}|).
		\end{aligned}
		$$
		By Gronwall's inequality,  we obtain the desired upper bound \eqref{eq:no blow-up}
		\begin{equation}\label{eq:Gronwall}
			1 + |x^{t}| + |\xi^{t}| \leq e^{2TC_{h,B}}(1 + |x^{s}| + |\xi^{s}|).
		\end{equation}
		 Moreover, since the backward-in-time flow $(x^{t},\xi^{t}) \mapsto (x^{s},\xi^{s})$ can be seen as a flow driven by the vector field $-\Phi_{h}$, a similar argument as above also gives a reversed bound
		\begin{equation}\label{eq:Gronwall backward}
			1 + |x^{s}| + |\xi^{s}| \leq e^{2TC_{h,B}}(1 + |x^{t}| + |\xi^{t}|).
		\end{equation}
		
		 For the volume-preserving property, we use the fact that a flow generated by a divergence-free vector field is volume-preserving, see Proposition 16.33 in \cite{Lee_Smooth}. For the vector field $\Phi_{h}$ we can compute
		$$
		\operatorname{div} \Phi_{h} = \sum_{j,k=1}^{d} \partial_{x_{j}}\partial_{\xi_{k}}h - \partial_{\xi_{j}}\partial_{x_{k}}h + B_{kj}(x)\partial_{\xi_{j}}\partial_{\xi_{k}}h = 0,
		$$
		so the flow $\chi_{B}(h)$ is volume-preserving. 
	\end{proof}
	
	The following lemma shows that the magnitude of the position and momentum vectors of certain magnetic flows are comparable in a time-averaged sense.    The lemma is inspired by \cite[Lemma 2.1]{Yajima_Schrodinger}, and is one of the crucial technical ingredients for dealing with unbounded magnetic potentials.  
	
	\begin{lem}\label{lem:t-averaging}
		 Let $\chi_{B}(h)$ be the magnetic Hamiltonian flow associated to a symbol $h$ satisfying Assumption \ref{asmp:h}  and a magnetic field $B$ satisfying $\|B_{jk}\|_\infty \leq C_{B}$. Let $I \subset \mathbb{R}$ be a compact interval  and fix an arbitrary $\epsilon > 0$.  We have for $(x^{t},\xi^{t}), t \in I$, satisfying  the flow $\chi_{B}(h)$ in \eqref{eq:def-hamilt-flow},
		\begin{equation}\label{eq:flow lem}
			\int_{I} \langle x^{\tau} \rangle^{-1-\epsilon} |\xi^{\tau}|\ d\tau \leq C_{h,B,\epsilon}(1+|I|), \qquad  (x,\xi) \in \mathbb{R}^{2d}.
		\end{equation}
	\end{lem}
	
	\begin{proof}
		We prove this in a similar way as  \cite[Lemma 2.1]{Yajima_Schrodinger}.  Again the $t$-dependence of $h(t)$ only makes a difference at the very end of the proof, so we suppress the notation for now.  By  Assumption \ref{asmp:h}, $h \in S^{0,(2)}(\mathbb{R}^{2d})$ and is second-order in $\xi$.   We thus have a lower bound for its  $\xi$-derivative, that is, there exist $b \in \mathbb{R}^{d}, C \in \mathbb{R}$ such that 
		\begin{equation}\label{eq:h xi-derivative}
			\begin{aligned}
				|\partial_{\xi}h(x,\xi)| \geq \big||\xi| - |b\cdot x + C|\big|.  
			\end{aligned}
		\end{equation}
		We denote a new constant $C_{h,B}$ based on $C_{h}$, $C_{B}$ and $|b|$ by 
		$$
		C_{h,B} := \max(|b|,C_{h}C_{B}).
		$$
		Without loss of generality, we can assume the constant $C_{h,B} \geq 1$, since otherwise the arguments below hold by replacing $C_{h,B}$ everywhere with $1$.  Sub-divide the interval $I = \sum_{j} I_{j}$ into compact intervals with length $|I_{j}| = T$, such that $T$ is sufficiently small satisfying
		\begin{equation}\label{eq:T small}
			200T^{2}C_{h,B}^{3}+60TC_{h,B}^{2} \leq \frac{1}{4}.
		\end{equation}
		It suffices to prove \eqref{eq:flow lem} separately for each $I_{j}$. We fix $I_{j}$ and suppress the index $j$, denoting it again by $I$. We also fix the starting point of the flow  as   $x,\xi \in \mathbb{R}^{d}$. Firstly, for the case of $TC_{h,B}|\xi^{t}| \leq 1 + |x^{t}|$ for all $t \in I$, the integral is bounded by
		\begin{equation}\label{eq:flow t-integral bound 1}
			\int_{I} \langle x^{\tau} \rangle^{-1-\epsilon} |\xi^{\tau}|\ d\tau \leq (TC_{h,B})^{-1}\int_{I} \langle x^{\tau} \rangle^{-1-\epsilon}(1+|x^{\tau}|)\ d\tau \leq C_{h,B}^{-1}.
		\end{equation}
		 For the other case, there exists $r \in I$ such that $TC_{h,B}|\xi^{r}| > 1 + |x^{r}|$.   We claim the following bounds hold for all $t \in I$:
		\begin{equation}\label{eq:t growth}
			1+|x^{t}| < 10TC_{h,B}|\xi^{r}|,\quad \tfrac{2}{3}|\xi^{r}| \leq |\xi^{t}| \leq \tfrac{3}{2}|\xi^{r}|,\quad t \in I.
		\end{equation}
		 We can in fact prove a stronger claim, that the above bounds are satisfied for all $t$ with $|t-r|\leq T$.  This claim can be verified by contradiction.   Denote $T^{*}: = \sup\{\sigma \leq T:  \eqref{eq:t growth} \text{ is satisfied for } |t-r|\leq \sigma\}$ and suppose $T^{*} < T$.    For  $t$  with $|t-r|\leq T^{*}$,  we can bound $1+|x^{t}|$ by 
		$$
		1+|x^{t}| \leq 1 + |x^{r}| + T^{*}\sup_{\tau\in I}\left|\frac{d|x^{\tau}|}{d\tau}\right|.
		$$
		Then by $TC_{h,B}|\xi^{r}| > 1 + |x^{r}|$,  the bound for $\frac{d|x^{t}|}{dt}$ in (\ref{eq:t-derivative bounds}) and the bound (\ref{eq:Gronwall}), we obtain  
		$$
		\begin{aligned}
			1 + |x^{t}| & \leq 1 + |x^{r}| + T^{*}C_{h,B}e^{2TC_{h,B}}(1+|x^{r}|+|\xi^{r}|)\\
			& < TC_{h,B}|\xi^{r}| + 3TC_{h,B}(TC_{h,B}|\xi^{r}|+|\xi^{r}|)\\
			& = 4TC_{h,B}|\xi^{r}| + 3T^{2}C_{h,B}^{2}|\xi^{r}|\\
			& < 9TC_{h,B}|\xi^{r}|,
		\end{aligned}
		$$
		where  we used $e^{2TC_{h,B}} < 3$ because by the smallness assumption \eqref{eq:T small} of $T$ we have $2TC_{h,B} < 1$. Similarly  by the bound for $\frac{d|\xi^{t}|}{dt}$ in \eqref{eq:t-derivative bounds} and the above, we have
		$$
		\begin{aligned}
			|\xi^{t}| & \leq |\xi^{r}| + T^{*}C_{h,B}e^{2TC_{h,B}}(1+|x^{r}|+|\xi^{r}|)\\
			& < |\xi^{r}| + (3T^{2}C_{h,B}^{2} + 3TC_{h,B})|\xi^{r}|\\
			& < \tfrac{11}{10}|\xi^{r}|,
		\end{aligned}
		$$
		where the smallness of $T$ as in \eqref{eq:T small} is also used to obtain the last inequality. The lower bound of $|\xi^{t}|$ holds analogously, that is
		$$
		|\xi^{t}|  > |\xi^{r}| - (3T^{2}C_{h,B}^{2} + 3TC_{h,B})|\xi^{r}| \geq \tfrac{9}{10}|\xi^{r}|.
		$$
		Because the flow $(x^{t},\xi^{t})$ is continuous in $t$, such $T^{*}$ cannot be maximal, thus contradicting its definition.
		
		Now for the constant $C$ in \eqref{eq:h xi-derivative},  we  distinguish two cases   based on the size of $|\xi^{r}|$. Firstly for the case that $|\xi^{r}| \leq 10|C|$, we directly have a desired bound
		$$
		\int_{I} \langle x^{\tau} \rangle^{-1-\epsilon} |\xi^{\tau}|\ d\tau \leq \int_{I} |\xi^{r}|\ d\tau \leq 10|C|T.
		$$
		For the other case $|\xi^{r}| > 10|C|$, using (\ref{eq:h xi-derivative}), (\ref{eq:t-derivative bounds}) and (\ref{eq:t growth}) we can calculate a lower bound for the second-order $t$-derivative of $|x^{t}|$. Here we will be taking $t$-derivatives on $h(t)$, so the $t$-dependence makes a difference, but things are under control since  by our Assumption \ref{asmp:h}, we have $\dot{h}(t) \in S^{0,(2)}(\mathbb{R}^{2d})$  uniformly  in $t$.  In particular, the assumption implies
		$$
		|\partial_{\xi}\dot{h}(t;x^{t},\xi^{t})| \leq C_{h}(1+|x^{t}|+|\xi^{t}|).
		$$
		Using bounds \eqref{eq:t-derivatives of flow} and \eqref{eq:t growth},  we have
		$$
			\begin{aligned}
				\dfrac{d^{2}|x^{t}|^{2}}{dt^{2}} & = 2|\partial_{\xi}h(t;x^{t},\xi^{t})|^{2} + 2\dfrac{d\partial_{\xi}h(t;x^{t},\xi^{t})}{dt}\cdot x^{t}\\
				& = 2|\partial_{\xi}h(t;x^{t},\xi^{t})|^{2} + 2\left(\partial_{\xi}\dot{h}(t;x^{t},\xi^{t}) + \partial_{x}\partial_{\xi}h(t;x^{t},\xi^{t})\dot{x}^{t} + \partial_{\xi}^{2}h(t;x^{t},\xi^{t})\dot{\xi}^{t}\right)\cdot x^{t}\\
				& \geq 2(\tfrac{9}{10}|\xi^{t}|-|b\cdot x^{t}|)^{2} - 2C_{h,B}(1+|x^{t}|+|\xi^{t}|)|x^{t}|\\
				& \geq 2(\tfrac{3}{5}|\xi^{r}|-10TC_{h,B}^{2}|\xi^{r}|)^{2} - 2C_{h,B}(10TC_{h,B}|\xi^{r}|+\tfrac{3}{2}|\xi^{r}|)10TC_{h,B}|\xi^{r}|\\
				& \geq \left(\tfrac{1}{2}-(200T^{2}C_{h,B}^{3}+60TC_{h,B}^{2})\right)|\xi^{r}|^{2},
			\end{aligned}
		$$
		so by the smallness of $T$ as in (\ref{eq:T small}), we have $
		\frac{d^{2}|x^{t}|^{2}}{dt^{2}} \geq \frac{1}{4}|\xi^{r}|^{2}$.   Then  we show for $\tau \in I$, $|x^{\tau}|$ can be bounded from below. Suppose at a time $\rho \in I$, $x^{\rho}$ attains the minimum value $|x^{\rho}| = \inf_{t\in I}|x^{t}|$. Then by a Taylor expansion of $|x^{\tau}|^{2}$ in the $t$-variable at $|x^{\rho}|^{2}$, we have  
		$$
		\begin{aligned}
			|x^{\tau}|^{2} & =|x^{\rho}|^{2} + \tfrac{d|x^{t}|^{2}}{dt}\Big|_{t=\rho}(\tau-\rho) + \tfrac{1}{2}\tfrac{d^{2}|x^{t}|^{2}}{dt^{2}}\Big|_{t=\theta}(\tau-\rho)^{2}\\
			& \geq |x^{\rho}|^{2} + \tfrac{1}{8}|\xi^{r}|^{2}(\tau-\rho)^{2},
		\end{aligned}
		$$
	where in the last inequality we used the fact that $|x^{\rho}|^{2}$ is a minimum, so $\tfrac{d|x^{t}|^{2}}{dt}\Big|_{t=\rho}(\tau-\rho) \geq 0$.    Then we have the desired bound for the integral 
		$$
		\begin{aligned}
			\int_{I} \langle x^{\tau} \rangle^{-1-\epsilon} |\xi^{\tau}|\ d\tau & \lesssim \int_{I} (1+|x^{\tau}|)^{-1-\epsilon} |\xi^{\tau}|\ d\tau\\
			& \lesssim \int_{I} (1+|\xi^{r}||\tau-\rho|)^{-1-\epsilon}|\xi^{r}|\ d\tau\\
			& \leq C'\epsilon^{-1},
		\end{aligned}
		$$
		 where we used the lower bound of $|x^{\tau}|$ and the fact that $|\xi^{\tau}|,|\xi^{\rho}|$ are comparable to $|\xi^{r}|$ by \eqref{eq:t growth}.  Combining this  with the bound (\ref{eq:flow t-integral bound 1}) for the other case, we obtain (\ref{eq:flow lem}).
		\end{proof}

	\section{Phase space approximation}

	 The evolution $u_{0} \mapsto u(t)$  satisfying (\ref{eq:electro-magnetic weyl schrodinger}) can be seen as a quantum evolution associated to the operator $\operatorname{Op}^{A}(h)$. Our goal in this section is to approximate such an evolution with a classical evolution associated to the Hamiltonian $h$ in the phase space. This is in some sense a manifestation of the quantum-classical correspondence principle, and the relevant classical evolution turns out to follow a magnetic Hamiltonian flow $\chi_{B}(h)$ in  phase space. So by taking a phase space perspective using the wavepacket transform $\mathcal{T}_{\lambda}^{A}$, we can obtain an approximate evolution in  phase space that is more transparent and tractable. This change of perspective reveals that to some extent, the quantum evolution $u_{0} \mapsto u(t)$ preserves phase space localizations in the sense of magnetic wavepackets, which leads to the magnetic modulation space boundedness of $u_{0} \mapsto u(t)$. 
	
	\subsection{Conjugating magnetic $\Psi$DOs}
	
The   first step in approximating the generalized magnetic Schr\"odinger evolution in  phase space is to approximate the effect of the operator $\operatorname{Op}^{A}(h)$ in  phase space.   
	That is, we need to understand how the conjugated operator $\mathcal{T}_{\lambda}^{A}\operatorname{Op}^{A}(h)\mathcal{T}_{\lambda}^{A*}$ behaves on phase space functions.  For real symbols $h$, the operator $\operatorname{Op}^{A}(h)$ is symmetric for the pairing of $u \in \mathcal{S}(\mathbb{R}^{d})$ and $g_{x,\xi}^{A,\lambda}$,
	$$
	\mathcal{T}_{\lambda}^{A}\operatorname{Op}^{A}(h)u(x,\xi) = \left\langle g_{x,\xi}^{A,\lambda},\operatorname{Op}^{A}(h)u\right\rangle_{L^{2}} = \left\langle \operatorname{Op}^{A}(h)g_{x,\xi}^{A,\lambda},u\right\rangle_{L^{2}},
	$$
	so it boils down to determining an approximation of the term $\operatorname{Op}^{A}(h)g_{x,\xi}^{A,\lambda}$ in  phase space. The following proposition relates differentiation and multiplication on the configuration space and the phase space.
	
	\begin{lem}\label{lem:substitution}
		For magnetic wavepackets $g_{x,\xi}^{A,\lambda}(y)$, multiplication and differentiation with respect to $y \in \mathbb{R}^{d}$ can be approximated by multiplication and differentiation with respect to $(x,\xi) \in \mathbb{R}^{2d}$, that is,
		\begin{equation}\label{eq:wavepacket differentiation}
			\begin{aligned}
				-i\partial_{y_{j}}g^{A,\lambda}_{x,\xi}(y) & = i\partial_{x_{j}}g^{A,\lambda}_{x,\xi}(y) - \sum_{k=1}^{d}B_{jk}(x)i\partial_{\xi_{k}}g^{A,\lambda}_{x,\xi}(y)\\
				& \quad + [A_{j}(y)-A_{j}(x) - r^{(1)}_{x}(A_{j}(y,x)) + r^{(1)}_{x}(A_{j}(x,y))]g^{A,\lambda}_{x,\xi}(y),
			\end{aligned}
		\end{equation}
		where $r^{(1)}_{x}$ denotes the remainder of the first-order Taylor expansion of $y$ at $x$. 
	\end{lem}
	
	\begin{proof}
		By  property  \eqref{eq:magnetic phase property} and  \eqref{eq:phase anti-symmetry} of $\varphi^{A}(y,x)$, we have
		$$
			\begin{gathered}
				\partial_{y_{j}} e^{i\varphi^{A}(y,x)} = i(A_{j}(y,0) - A_{j}(y,x))e^{i\varphi^{A}(y,x)},\\
				\partial_{x_{j}}  e^{i\varphi^{A}(y,x)} = i(- A_{j}(x,0) + A_{j}(x,y))e^{i\varphi^{A}(y,x)}.
			\end{gathered}
		$$
		A direct computation gives
		$$
		-i\partial_{y_{j}}g^{A,\lambda}_{x,\xi}(y) = i\partial_{x_{j}}g^{A,\lambda}_{x,\xi}(y) + [A_{j}(y)-A_{j}(x)-A_{j}(y,x)+A_{j}(x,y)]g^{A,\lambda}_{x,\xi}(y).
		$$
		Next, we use Taylor's formula to expand the term $-A_{j}(y,x)+A_{j}(x,y)$ with respect to the $y$ variable at the value $x$. Recall the definition of the potential
		$$
		A_j(y, x):=-\sum_{k=1}^d \int_0^1 s\left(y_k-x_k\right) B_{j k}(x+s(y-x)) d s,
		$$
		so evaluating $A_{j}$ at $y=x$ we have $A_{j}(x,x) = 0$. The first-order derivatives of $A_{j}$ are
		$$
		\begin{aligned}
			\partial_{y_{l}}A_{j}(y,x) = -\int^{1}_{0} sB_{jl}(x+s(y-x)) d s -\sum_{k=1}^d \int_0^1 s^{2}\left(y_k-x_k\right) \partial_{l}B_{j k}(x+s(y-x)) d s.
		\end{aligned}
		$$
		The derivatives at $y=x$ evaluate to
		\begin{equation}\label{eq:A_j^(1)}
			\partial_{y_{l}}A_{j}(x,x) = -\int^{1}_{0} sB_{jl}(x) d s = - \frac{1}{2} B_{jl}(x). 
		\end{equation}
		Further, the second-order derivatives are, for $1 \leq m,l\leq d$, 
		\begin{equation}\label{eq:A_j^(2)}
			\begin{aligned}
				\partial_{y_{m}}\partial_{y_{l}}A_{j}(y,x) & = -\int^{1}_{0}  s^{2}  \partial_{m}B_{jl}(x+s(y-x)) d s - \int_0^1 s^{2} \partial_{l}B_{j m}(x+s(y-x)) d s\\
				& \quad - \sum_{k=1}^d \int_0^1 s^{3}\left(y_k-x_k\right) \partial_{m}\partial_{l}B_{j k}(x+s(y-x)) d s,
			\end{aligned}
		\end{equation}
		so by \eqref{eq:A_j^(1)} and \eqref{eq:A_j^(2)}, Taylor's formula of $A_{j}$ takes the form
		$$
		A_{j}(y,x) = - \sum^{d}_{l=1}\frac{1}{2} B_{jl}(x)(y_{l}-x_{l}) + r^{(1)}_{x}(A_{j}(y,x)),
		$$
		where,  for  some $\theta \in (0,1)$, the remainder can be expressed by
		\begin{equation}\label{def:magnetic remainder}
			\begin{aligned}
				r^{(1)}_{x}(A_{j}(y,x)) & := \sum_{m,l=1}^{d} (y_{m}-x_{m})(y_{l}-x_{l}) \partial_{y_{m}}\partial_{y_{l}}A_{j}(x+\theta(y-x),x).
			\end{aligned}
		\end{equation}
		Expanding also $A_{j}(x,y)$ and combining the terms, we obtain
		$$
			- A_{j}(y,x) + A_{j}(x,y) = \sum_{l=1}^{d} B_{jl}(x)(y_{l}-x_{l}) - r^{(1)}_{x}(A_{j}(y,x)) + r^{(1)}_{x}(A_{j}(x,y)).
		$$
		The lemma then follows from the identity $\partial_{\xi_{j}}g^{A,\lambda}_{x,\xi}(y) = i(y_{j}-x_{j})g^{A,\lambda}_{x,\xi}(y)$.
	\end{proof}
	
  We build upon the above lemma to show that the action of $\operatorname{Op}^{A}(h)$ on wavepackets $g^{A,\lambda}_{x,\xi}$ can be approximated, up to some error terms, as a phase space differentiation along the magnetic Hamiltonian vector field associated to $h$. The lemma below can hold for more general $h$, but for  simplicity we formulate it for the most relevant symbol class $S^{0,(2)}(\mathbb{R}^{2d})$.
	
	\begin{lem}\label{lem:flat approximation}
 Let $h \in S^{0,(2)}(\mathbb{R}^{2d})$. 		The magnetic pseudo-differentiation $\operatorname{Op}^{A}(h)$ on wavepackets $g_{x,\xi}^{A,\lambda}$ can be approximated by a phase space differentiation $\tilde{H}^{B}$. That is, 
		\begin{equation}\label{eq:conjugation calculation}
			\begin{aligned}
				\operatorname{Op}^{A}(h)g_{x,\xi}^{A,\lambda}(y) = \left(i\tilde{H}^{B} +\mathcal{R}\right)g^{A,\lambda}_{x,\xi}(y), \qquad  y \in \R^d,  
			\end{aligned}
		\end{equation}
		where the components of $\tilde{H}^{B}$ are
		$$
		\tilde{H}^{B}_{j} := \left(\partial_{\xi_{j}}h\partial_{x_{j}},-\partial_{x_{j}}h\partial_{\xi_{j}}+\sum_{k=1}^{d}B_{kj}(x)\partial_{\xi_{k}}h\partial_{\xi_{j}}\right),\quad 1\leq j \leq d.
		$$
		The approximation error $\mathcal{R}$ term consists of the three parts
		$$
		\mathcal{R} = m^{A}(h) + \mathcal{R}_{1} + \mathcal{R}_{2},
		$$
 as defined in  (\ref{def:m_h}), (\ref{def:R_1}), and (\ref{def:R_2}),   where $m^{A}(h)$ denotes a multiplier that only depends on the phase space variables $x, \xi$, while $\mathcal{R}_{1}$, $\mathcal{R}_{2}$ denotes respectively a multiplicative operator and a magnetic pseudo-differential operator on the $y$ variable.
	\end{lem}

	\begin{proof}
		For  convenience of later proofs, we first convert to the magnetic Kohn-Nirenberg quantization $\operatorname{Op}_{KN}^{A}(h^{*})=\operatorname{Op}^{A}(h)$. By Lemma \ref{lem:Weyl to KN}, the symbol $h^{*} \in S^{0,(2)}(\mathbb{R}^{2d})$ is given by 
		\begin{equation}\label{def:KN symbol}
			h^{*}(y,\eta) := h(y,\eta) +  h_{r}(y,\eta),\quad h_{r} \in S^{0,(0)}(\mathbb{R}^{2d}),
		\end{equation}
		and we can represent $\operatorname{Op}^{A}(h)$ using the Kohn-Nirenberg quantization 
		$$
		\operatorname{Op}^{A}\left(h\right)g^{A,\lambda}_{x,\xi}(y) =  \left(\operatorname{Op}^{A}_{KN}\left(h\right) + \operatorname{Op}^{A}_{KN}\left(h_{r}\right)\right)g^{A,\lambda}_{x,\xi}(y).
		$$
		The $\operatorname{Op}^{A}_{KN}\left(h_{r}\right)$ will be included in the pseudo-differential remainder $\mathcal{R}_{2}$, since its symbol $h_{r}$ is nicely bounded. Our focus here is to approximate $\operatorname{Op}^{A}_{KN}\left(h\right)$, and we start with a Taylor expansion on $h$ at the point $(x,\xi)$, that is,
		\begin{equation}\label{eq:expansion part 1}
			\begin{aligned}
				& \quad \operatorname{Op}^{A}_{KN}(h)g^{A,\lambda}_{x,\xi}(y)\\
				& = \operatorname{Op}_{KN}^{A}\left[h(x,\xi) + \partial_{y}h(x,\xi)\cdot(y-x) + \partial_{\eta}h(x,\xi)\cdot(\eta-\xi) + r^{(1)}_{x,\xi}(h)(y,\eta)\right]g^{A,\lambda}_{x,\xi}(y)\\
				& = \operatorname{Op}_{KN}^{A}\left[ \partial_{y}h(x,\xi)\cdot(y-x) + \partial_{\eta}h(x,\xi)\cdot\eta + r^{(1)}_{x,\xi}(h)(y,\eta)\right]g^{A,\lambda}_{x,\xi}(y)\\
				& \quad + (h(x,\xi) - \partial_{\eta}h(x,\xi)\cdot \xi)g^{A,\lambda}_{x,\xi}(y).
			\end{aligned}
		\end{equation}
		Here $r^{(1)}_{x,\xi}(h)$ denotes the remainder term from the first-order expansion:
		\begin{equation}\label{def:r1 remainder}
			\begin{aligned}
				r^{(1)}_{x,\xi}(h)(y,\eta) & := \int_{0}^{1}(1-s) \partial_s^{2} h(x+s(y-x), \xi+s(\eta-\xi))\ ds.
			\end{aligned}
		\end{equation}
		By the basic quantizations (\ref{eq:basic quantization}) and the identity (\ref{eq:wavepacket differentiation}) in Lemma \ref{lem:substitution}, we can calculate two magnetic $\Psi$DO terms on the right hand side of (\ref{eq:expansion part 1}) as
		\begin{equation}\label{eq:expansion part 2}
			\begin{aligned}
				\operatorname{Op}_{KN}^{A}\left[\partial_{y}h(x,\xi)\cdot(y-x)\right]g^{A,\lambda}_{x,\xi}(y) & = \partial_{y}h(x,\xi)\cdot(y-x)g^{A,\lambda}_{x,\xi}(y)\\
				& = -i\partial_{y}h(x,\xi)\cdot \partial_{\xi}g^{A,\lambda}_{x,\xi}(y),
			\end{aligned}
		\end{equation}
		and
		\begin{equation}\label{eq:expansion part 3}
			\begin{aligned}
				& \operatorname{Op}_{KN}^{A}\left[\partial_{\eta}h(x,\xi)\cdot\eta\right]g^{A,\lambda}_{x,\xi}(y)\\
				&\quad = \partial_{\eta}h(x,\xi)\cdot(-i\partial_{y} - A(y))g^{A,\lambda}_{x,\xi}(y)\\
				&\quad = \sum_{j=1}^{d} \partial_{\eta_{j}}h(x,\xi)\left(i\partial_{x_{j}}g^{A,\lambda}_{x,\xi}(y) - \sum_{k=1}^{d}B_{jk}(x)i\partial_{\xi_{k}}g^{A,\lambda}_{x,\xi}(y)\right) + \\
				&\quad \quad \quad + \sum_{j=1}^{d} \partial_{\eta_{j}}h(x,\xi)\left(-A_{j}(x) - r^{(1)}_{x}(A_{j}(y,x)) + r^{(1)}_{x}(A_{j}(x,y))\right)g^{A,\lambda}_{x,\xi}(y).
			\end{aligned}
		\end{equation}
		Then substituting (\ref{eq:expansion part 2}), (\ref{eq:expansion part 3}) into (\ref{eq:expansion part 1}) yields
		$$
		\begin{aligned}
			& \operatorname{Op}^{A}(h) g^{A,\lambda}_{x,\xi}(y)\\
			&\quad = i\sum_{j=1}^{d}\partial_{\eta_{j}}h(x,\xi)\partial_{x_{j}}g^{A,\lambda}_{x,\xi}(y) -i  \sum_{j=1}^{d} \left(\partial_{y_{j}}h(x,\xi) - \sum_{k=1}^d B_{kj}(x)\partial_{\eta_{k}}h(x,\xi)\right)  \partial_{\xi_{j}}g^{A,\lambda}_{x,\xi}(y) +\\
			&\quad \quad + \left(m^{A}(h) + \mathcal{R}_{1} + \mathcal{R}_{2}\right)g^{A,\lambda}_{x,\xi}(y)
		\end{aligned}
		$$
		with a multiplicative factor $m^{A}(h)$ which only depends on the variable $x,\xi$
		\begin{equation} \label{def:m_h}
			\begin{aligned}
				m^{A}(h)(x,\xi) := h(x,\xi) - \partial_{\eta}h(x,\xi)\cdot (\xi+A(x)),
			\end{aligned}
		\end{equation}
		and two other error terms involving the $y$ variable, including one multiplicative $\mathcal{R}_{1}$,
		\begin{equation}\label{def:R_1}
			\begin{aligned}
				\mathcal{R}_{1}g^{A,\lambda}_{x,\xi}(y) & :=  \partial_{\eta}h(x,\xi)\cdot\left( - r^{(1)}_{x}(A(y,x)) + r^{(1)}_{x}(A(x,y))\right)g^{A,\lambda}_{x,\xi}(y)
			\end{aligned}
		\end{equation}
		and one magnetic pseudo-differential $\mathcal{R}_{2}$,
		\begin{equation}\label{def:R_2}
			\mathcal{R}_{2}g^{A,\lambda}_{x,\xi}(y) := \operatorname{Op}_{KN}^{A}\left(r^{(1)}_{x,\xi}(h) + h_{r}\right)g^{A,\lambda}_{x,\xi}(y).
		\end{equation}
	\end{proof}
	
	We observe that in the above lemma, the phase space vector field $\tilde{H}^{B}$ coincides with the Hamiltonian vector field of the  magnetic  flow $\chi_{B}(h)$ in (\ref{eq:magnetic hamiltonian flow flat}).  This makes it possible for us to relate the quantum evolution associated to $\operatorname{Op}^{A}(h)$ with the classical flow $\chi_{B}(h)$ associated to $h$.

	\subsection{Propagation along the phase space flow}
	
	Based on the representation (\ref{eq:conjugation calculation}) in Lemma \ref{lem:flat approximation},  we see that  if a time-dependent wavepacket $w$ satisfies 
	$$
	D_{t}w(t;x,\xi) + (i\tilde{H}^{B} + m^{A}(h) + \mathcal{R})w(t;x,\xi) = 0, \quad w(t;x,\xi) = g^{A,\lambda}_{x,\xi},
	$$
	on the phase space $\mathbb{R}^{2d}$, then it will also satisfy
	$$
	(D_{t} + \operatorname{Op}^{A}(h))w(t,x,\xi;y) = 0
	$$
	 on the configuration space $\mathbb{R}^{d}$.  
	A solution $u(t)$ to (\ref{eq:electro-magnetic weyl schrodinger}) can then be constructed as a superposition of such wavepackets $w$    with the initial data $\tilde{u}_{0} := \mathcal{T}^{A}_{\lambda}u_{0}$, that is 
	$$
	u(t;y) = \int_{\mathbb{R}^{2d}} w(t,y;x,\xi)\tilde{u}_{0}(x,\xi) \ dxd\xi.
	$$
	However,  the existence of such $w$ is unclear  due to the complicated error term $\mathcal{R}w$. We instead first show the existence of a wavepacket solution to a simpler equation.
	
	\begin{lem}\label{lem:wavepacket solution}
		For $h$ satisfying Assumption \ref{asmp:h}, with associated magnetic Hamiltonian vector field $\tilde{H}^{B}$ and $m^{A}(h)$ defined  in \eqref{def:m_h}, there exists a wavepacket $w \in C(\mathbb{R},C^{\infty}(\mathbb{R}^{2d}))$, which satisfies the equation
		$$
			\begin{cases}
				 D_{t}w(t;x,\xi) + (i\tilde{H}^{B} + m^{A}(h))w(t;x,\xi) = 0,\\
				w(0;x,\xi) = g^{A,\lambda}_{x,\xi}.
			\end{cases}
		$$
	\end{lem}
	
	\begin{proof}
		By Lemma \ref{lem:flow}, the magnetic Hamiltonian flow $\chi_{B}(h)(t,0): (x,\xi) \mapsto (x^{t},\xi^{t})$ associated to $h$ exists globally, and we can construct a time-dependent wavepacket by setting
		$$
		g^{A,\lambda}_{x^{t},\xi^{t}}(y) := \chi_{B}(h)(t,0)g^{A,\lambda}_{x,\xi}(y),\quad y \in \mathbb{R}^{d}.
		$$
		Then this wavepacket $g^{A,\lambda}_{x^{t},\xi^{t}}$ satisfies
		$$
		\partial_{t}g^{A,\lambda}_{x^{t},\xi^{t}}(y) = \tilde{H}^{B}g^{A,\lambda}_{x^{t},\xi^{t}}(y),\quad g^{A,\lambda}_{x^{t},\xi^{t}}\big|_{t=0}(y) = g^{A,\lambda}_{x,\xi}(y).
		$$
		We can also construct a phase term $e^{i\psi(t)}$, with a phase function $\psi$ based on $m^{A}(h)$, by
		$$
		\psi(t;x,\xi) := \int_{0}^{t} -m^{A}(h)(\tau;x^{\tau},\xi^{\tau})\ d\tau,
		$$
		so that $\psi(t)$ satisfies
		$$
		\partial_{t}\psi(t) = -m^{A}(h)(t;x^{t},\xi^{t}),\quad e^{i\psi(t)}\big|_{t=0} = 1.
		$$
	We	modulate the wavepacket with the phase term $e^{i\psi(t)}$ and   define 
		$$
			g^{A,\lambda}_{\psi}(t;y,x,\xi) := e^{i\psi(t)}g^{A,\lambda}_{x^{t},\xi^{t}}(y).
		$$
	The modulated wavepacket $g^{A,\lambda}_{\psi}$ now satisfies
		\begin{equation}\label{eq:t-derivative}
			\begin{aligned}
				\partial_{t}g^{A,\lambda}_{\psi}(t;y,x,\xi) & = i\partial_{t}\psi(t)e^{i\psi(t)}g^{A}_{x^{t},\xi^{t}}(y) + e^{i\psi(t)}\partial_{t}g^{A,\lambda}_{x^{t},\xi^{t}}(y)\\
				& = -im^{A}(h)e^{i\psi(t)}g^{A,\lambda}_{x^{t},\xi^{t}}(y) + e^{i\psi(t)}\tilde{H}^{B}g^{A,\lambda}_{x^{t},\xi^{t}}(y)\\
				& = \left(-im^{A}(h) + \tilde{H}^{B}\right)g^{A,\lambda}_{\psi}(t;y,x,\xi),
			\end{aligned}
		\end{equation}
		with $g^{A,\lambda}_{\psi}(0;y,x,\xi) = g^{A,\lambda}_{x,\xi}(y)$ as desired.
	\end{proof}
	
	Using the wavepacket solution $g^{A,\lambda}_{\psi}(t)$ from the above lemma, we can construct a parametrix to  equation (\ref{eq:electro-magnetic weyl schrodinger}).
	
	\begin{prop}\label{prop:parametrix}
		Let $0\leq s,t \leq T$. We define the operator $\tilde{S}(t,s)$ by 
		\begin{equation}\label{def:S tilde}
			\tilde{S}(t,s)u_{0}(y) := \int_{\mathbb{R}^{2d}\times\mathbb{R}^{d}} g^{A,\lambda}_{\psi}(t;y,x,\xi) \overline{g^{A,\lambda}_{\psi}}(s;y',x,\xi)u_{0}(y')\ dy'dxd\xi,
		\end{equation}
		for $u_{0} \in \mathcal{S}(\mathbb{R}^{d})$. We also define the operator $K(t,s)$ by
		\begin{equation}\label{def:K}
			K(t,s)u_{0}(y) := \int (\mathcal{R}_{1}+\mathcal{R}_{2})g^{A,\lambda}_{\psi}(t;y,x,\xi) \overline{g^{A,\lambda}_{\psi}}(s;y',x,\xi)u_{0}(y')\ dy'dxd\xi,
		\end{equation}
		for $u_{0} \in \mathcal{S}(\mathbb{R}^{d})$. We then have
		\begin{equation}\label{eq:parametrix equation}
			(D_{t}+\operatorname{Op}^{A}(h))\tilde{S}(t,s)u_{0} = K(t,s)u_{0},\quad u_{0} \in \mathcal{S}(\mathbb{R}^{d}).
		\end{equation}
		Moreover,  for  $1 \leq p \leq \infty,  m \in \mathbb{R}$ and fixed $t,s$,  the operator $\tilde{S}(t,s)$ is bounded on $M^{m,p}_{A}(\mathbb{R}^{d})$. The operators $\tilde{S}(t,\cdot)$ and $\tilde{S}(\cdot,s)$ are also strongly continuous on $M^{m,p}_{A}(\mathbb{R}^{d})$. 
	\end{prop}
	
	\begin{proof}
		 By construction of  $\tilde{S}$ based on  the wavepacket solution $g^{A,\lambda}_{\psi}(t)$, we can calculate using Lemma \ref{lem:flat approximation} and (\ref{eq:t-derivative})
		$$
			\begin{aligned}
				\operatorname{Op}^{A}(h) g^{A,\lambda}_{\psi}(t;y,x,\xi) & = e^{i\psi(t)}(i\tilde{H}^{B} + \mathcal{R}) g^{A,\lambda}_{x^{t},\xi^{t}}(y)\\
				& = (i\tilde{H}^{B} + m^{A}(h) + \mathcal{R}_{1}+\mathcal{R}_{2})g^{A,\lambda}_{\psi}(t;y,x,\xi),
			\end{aligned}
		$$
		and
		$$
		D_{t}(g^{A,\lambda}_{\psi})(t;y,x,\xi) = -(m^{A}(h)+i\tilde{H}^{B})g^{A,\lambda}_{\psi}(t;y,x,\xi).
		$$
		Then the equation (\ref{eq:parametrix equation}) can be verified based on the above calculations 
		$$
		\begin{aligned}
			& (D_{t}+\operatorname{Op}^{A}(h))\tilde{S}(t,s)u_{0}(y)\\
			&\quad = \int_{\mathbb{R}^{2d}\times\mathbb{R}^{d}} (D_{t}+\operatorname{Op}^{A}(h))(g^{A,\lambda}_{\psi})(t;y,x,\xi) \overline{g^{A,\lambda}_{\psi}}(s;y',x,\xi)u_{0}(y')\ dy'dxd\xi\\
			&\quad = \int_{\mathbb{R}^{2d}\times\mathbb{R}^{d}} (\mathcal{R}_{1}+\mathcal{R}_{2})(g^{A,\lambda}_{\psi})(t;y,x,\xi) \overline{g^{A,\lambda}_{\psi}}(s;y',x,\xi)u_{0}(y')\ dy'dxd\xi\\
			&\quad = K(t,s)u_{0}(y).
		\end{aligned}
		$$
		
		For the  $M^{m,p}_{A}(\mathbb{R}^{d})$ boundedness of $\tilde{S}(t,s)$, in view of Lemma \ref{lem:weighted kernel bound}, it suffices to show the phase space kernel of $\tilde{S}(t,s)$ has sufficient off-diagonal decay of the form $\langle \lambda(z-x^{t})\rangle^{-N}\langle \lambda^{-1}(\zeta-\xi^{t})\rangle^{-N'}$.  Note that  by definition the $M^{m,p}_{A}(\mathbb{R}^{d})$ norm of $\tilde{S}(t,s)u_{0}$  equals the weighted $L^{p}(\mathbb{R}^{2d})$ norm on the phase space
		$$
		\begin{aligned}
			\|\tilde{S}(t,s)u_{0}\|_{M^{m,p}_{A}(\mathbb{R}^{d})} & = \|\nu_{\lambda,m}\mathcal{T}^{A}_{\lambda}\tilde{S}(t,s)u_{0}\|_{L^{p}(\mathbb{R}^{2d})}.
		\end{aligned}
		$$
		 We fully write out the wavepacket transform $\mathcal{T}^{A}_{\lambda}\tilde{S}(t,s)u_{0}$ to see that 
		$$
		\begin{aligned}
			& \mathcal{T}^{A}_{\lambda}\tilde{S}(t,s)u_{0}(z,\zeta)\\
			&\quad = \int_{\mathbb{R}^{2d}} \int_{\mathbb{R}^{d}} \overline{g^{A,\lambda}_{z,\zeta}}(y)g^{A,\lambda}_{x^{t},\xi^{t}}(y)\ dy \left(e^{i\psi(t)-i\psi(s)} \int_{\mathbb{R}^{d}} \overline{g^{A,\lambda}_{x^{s},\xi^{s}}}(y')u_{0}(y')\ dy'\right)\ dxd\xi.
		\end{aligned}
		$$
		Denote by $G^{A,\lambda}$ the integral kernel
		\begin{equation}\label{def:G^A,lambda}
			G^{A,\lambda}(z,\zeta,x^{t},\xi^{t}) := \int_{\mathbb{R}^{d}} \overline{g^{A,\lambda}_{z,\zeta}}(y)g^{A,\lambda}_{x^{t},\xi^{t}}(y)\ dy.
		\end{equation}
		Then $	\mathcal{T}^{A}_{\lambda}\tilde{S}(t,s)u_{0}$ can be seen as associated to the phase space kernel $G^{A,\lambda}$ 
		$$
		\mathcal{T}^{A}_{\lambda}\tilde{S}(t,s)u_{0}(z,\zeta) = \int_{\mathbb{R}^{2d}} G^{A,\lambda}(z,\zeta,x^{t},\xi^{t}) e^{i\psi(t)-i\psi(s)} \tilde{u}_{0}(x^{s},\xi^{s})\ dxd\xi,
		$$
		 We can then write out the weighted version as
		$$
		\begin{aligned}
			& \nu_{\lambda,m}( z,\zeta)\mathcal{T}^{A}_{\lambda}\tilde{S}(t,s)u_{0}(z,\zeta)\\
			= & \int_{\mathbb{R}^{2d}} \nu_{\lambda,m}( z,\zeta)G^{A,\lambda}(z,\zeta,x^{t},\xi^{t}) \nu_{\lambda,-m}( x^{s},\xi^{s})e^{i\psi(t)-i\psi(s)} \nu_{\lambda,m}( x^{s},\xi^{s})\tilde{u}_{0}(x^{s},\xi^{s})\ dxd\xi,
		\end{aligned}
		$$
		and note the weighted phase space kernel is related to the unweighted $G^{A,\lambda}$ by
		\begin{equation}\label{eq:G^A,lambda weighted}
			\begin{aligned}
				& |\nu_{\lambda,m}( z,\zeta)G^{A,\lambda}(z,\zeta,x^{t},\xi^{t}) \nu_{\lambda,-m}(x^{s},\xi^{s})|\\
				\lesssim & \langle \lambda^{-1}(z-x^{t})\rangle^{m}\langle \lambda^{-1}(\zeta-\xi^{t})\rangle^{m}\nu_{\lambda,m}( x^{t},\xi^{t})\nu_{\lambda,-m}( x^{s},\xi^{s})|G^{A,\lambda}(z,\zeta,x^{t},\xi^{t})|\\
				\lesssim & C_{T,h,B}^{m}\langle \lambda^{-1}(z-x^{t})\rangle^{m}\langle \lambda^{-1}(\zeta-\xi^{t})\rangle^{m}|G^{A,\lambda}(z,\zeta,x^{t},\xi^{t})|,
			\end{aligned}
		\end{equation}
		where we used the flow property \eqref{eq:Gronwall} to bound $\nu_{\lambda,m}( x^{t},\xi^{t})\nu_{\lambda,-m}( x^{s},\xi^{s}) \lesssim C_{T,h,B}^{m}$.
		
		 By the volume-preservation property in Lemma \ref{lem:flow}, we can replace $dxd\xi$ with $dx^{s}d\xi^{s}$, so that 
		$$
		\begin{aligned}
			\|\nu_{\lambda,m}e^{i\psi(t)-i\psi(s)} \tilde{u}_{0}\|_{L^{p}(\mathbb{R}^{2d})} & = \left(\int_{\mathbb{R}^{2d}} |\nu_{\lambda,m}( x^{s},\xi^{s})e^{i\psi(t)-i\psi(s)} \tilde{u}_{0}(x^{s},\xi^{s})|^{p}\ dx^{s}d\xi^{s}\right)^{1/p}\\
			& = \|u_{0}\|_{M^{m,p}_{A}(\mathbb{R}^{d})}.
		\end{aligned}
		$$
		So by Young's inequality for integral operators and the relation \eqref{eq:G^A,lambda weighted}, the $M^{m,p}_{A}(\mathbb{R}^{d})$ boundedness of $\tilde{S}(t,s)$ will be implied by the  off-diagonal decay of the kernel
		\begin{equation}\label{eq:G^A,lambda bound}
			|G^{A,\lambda}(z,\zeta,x^{t},\xi^{t})| \lesssim \langle \lambda(z-x^{t})\rangle^{-N} \langle \lambda^{-1}(\zeta-\xi^{t})\rangle^{-N'}, 
		\end{equation}
 for sufficiently large $N,N' \in \mathbb{N}$. 
		To show the above bound,  we notice that based on (\ref{def:magnetic flux}), the product of the two wavepackets $\overline{g^{A,\lambda}_{z,\zeta}}$ and $g^{A,\lambda}_{x^{t},\xi^{t}}$ can be  complemented with a phase factor $e^{i\varphi^{A}(z,x^{t})}$ to make a factor $e^{i\Gamma^{B}}$ appear. That is, 
		$$
			\begin{aligned}
				\quad &  e^{-i\varphi^{A}(z,x^{t})}e^{-i\zeta\cdot(z-x^{t})}\overline{g^{A,\lambda}_{z,\zeta}}(y) g^{A,\lambda}_{x^{t},\xi^{t}}(y)\\
				& = \lambda^{d} e^{i (\xi^{t}-\zeta) \cdot (y-x^{t})}  e^{-i \varphi^A(y, z)} e^{i \varphi^A(y, x^{t})} e^{-i\varphi^{A}(z,x^{t})} g(\lambda(y-z))g(\lambda(y-x^{t})) \\
				& = \lambda^{d} e^{i (\xi^{t}-\zeta) \cdot (y-x^{t})}  e^{i\Gamma^{B}(y,x^{t},z)} g(\lambda(y-z))g(\lambda(y-x^{t})).
			\end{aligned}
		$$
		We can then make use of the derivative bounds (\ref{eq:flux bound}) of  $e^{i\Gamma^{B}}$,   which yield  for $\lambda \geq 1$
		$$
			|\partial_{\lambda y}^{\alpha} e^{i\Gamma^{B}}(y,x^{t},z)| \leq C_{B,\alpha}(\langle \lambda(y-x^{t})\rangle + \langle \lambda(y-z)\rangle)^{|\alpha|},\qquad  |\alpha|\geq 1.
		$$
	 Denote part of the wavepacket product by
		\begin{equation}\label{def:wavepackets product W}
			W_{B,\lambda}(y,x^{t},z) := e^{i\Gamma^{B}(y,x^{t},z)} g(\lambda(y-z))g(\lambda(y-x^{t})).
		\end{equation}
		By the derivative bounds  of $e^{i\Gamma^{B}}$ above  and $g \in \mathcal{S}(\mathbb{R}^{d})$, we see that the partial derivatives of $W_{B,\lambda}$ with respect to $\lambda y$ can be bounded for all $N,N' \in \mathbb{N}$ by
		$$
			|\partial_{\lambda y}^{\alpha}W_{B,\lambda}(y,x^{t},z)| \leq C_{N,N',\alpha,B} \langle \lambda(y-z) \rangle^{-N} \langle \lambda(y-x^{t}) \rangle^{-N'}.
		$$
		Using $\langle \lambda(y-z)\rangle^{-N} \lesssim \langle \lambda(y-x^{t})\rangle^{N}\langle \lambda(z-x^{t})\rangle^{-N}$, we further obtain for all $N,N',N'' \in \mathbb{N}$
			\begin{equation}\label{eq:W derivatives}
			\begin{aligned}
				& |\partial_{\lambda y}^{\alpha}W_{B,\lambda}(y,x^{t},z)|\\
				& \quad \lesssim C_{N,N',N'',\alpha,B} \langle \lambda(z-x^{t})\rangle^{-N} \langle \lambda(y-z) \rangle^{-N'} \langle \lambda(y-x^{t}) \rangle^{-N''}.
			\end{aligned}
		\end{equation}
		We now re-write $G^{A,\lambda}$ with $W_{B,\lambda}$ as
		\begin{equation}\label{eq:G^A,lambda re-write}
			\begin{aligned}
				& G^{A,\lambda}(z,\zeta,x^{t},\xi^{t})\\
				& \quad = e^{i\varphi^{A}(z,x^{t})}e^{i\zeta\cdot(z-x^{t})} \int_{\mathbb{R}^{d}} e^{i (\xi^{t}-\zeta) \cdot (y-x^{t})} W_{B,\lambda}(y,x^{t},z)\ d\lambda y,
			\end{aligned}
		\end{equation}
		 and we can obtain a decay factor $\langle \lambda^{-1}(\zeta-\xi^{t})\rangle^{-N'}$ by integration by parts repeatedly with respect to the $\lambda y$-integral, relying on   the identity
		$$
			\langle \lambda^{-1}(\zeta-\xi^{t})\rangle^{-2}(1-\Delta_{\lambda y})  e^{i\lambda^{-1}(\xi^{t}-\zeta) \cdot \lambda y} = e^{i(\xi^{t}-\zeta) \cdot y}.
		$$
		Now combined with the $W_{B,\lambda}$ bounds \eqref{eq:W derivatives}, we see by integration by parts
		$$
		\langle \lambda(z-x^{t})\rangle^{N}\langle \lambda^{-1}(\zeta-\xi^{t})\rangle^{N'}|G^{A,\lambda}(z,\zeta,x^{t},\xi^{t})| \lesssim C_{B,N,N'}
		$$
		as desired.
		
		 For the strong continuity of $\tilde{S}$, by the fact that $\tilde{S}(t,s)$ is continuous on $M^{m,p}_{A}(\mathbb{R}^{d})$ for fixed $t,s$, it suffices to show the strong continuity on the dense subspace $\mathcal{S}(\mathbb{R}^{d})$ of $M^{m,p}_{A}(\mathbb{R}^{d})$.  Here we show the strong continuity for $\tilde{S}(\cdot,s)$, and the strong continuity of $\tilde{S}(t,\cdot)$ is analogous.  Since the flow $\chi_{B}(t,s)$ is continuous in $t$, and the wavepackets $g^{A,\lambda}_{\psi}$ are dominated by the Gaussian $g$, that is 
		$$
		|g^{A,\lambda}_{\psi}(\tau;y,x,\xi)| \leq |g(\lambda(y-x))|,
		$$
		we can use the dominated convergence theorem to see at the  limit $t \to s$, we recover the inversion formula in Proposition \ref{prop:magnetic modulation space}. That is,
		$$
		\begin{aligned}
			\lim_{t \to s} \tilde{S}(t,s)u_{0}(y) & =  \int_{\mathbb{R}^{2d}\times\mathbb{R}^{d}} \lim_{t \to s} g^{A,\lambda}_{\psi}(t;y,x,\xi) \overline{g^{A,\lambda}_{\psi}}(s;y',x,\xi)u_{0}(y')\ dy'dxd\xi\\
			& = \int_{\mathbb{R}^{2d}\times\mathbb{R}^{d}}  g^{A,\lambda}_{x^{s},\xi^{s}}(y) \overline{g^{A,\lambda}_{x^{s},\xi^{s}}}(y')u_{0}(y')\ dy'dx^{s}d\xi^{s}\\
			& = u_{0}(y),
		\end{aligned}
		$$
		so the strong continuity of $\tilde{S}$ holds for all $u_{0} \in \mathcal{S}(\mathbb{R}^{d})$ and by continuity of $\tilde{S}(t,s)$ and density of $\mathcal{S}(\mathbb{R}^{d})$ in $M^{m,p}_{A}(\mathbb{R}^{d})$ it holds on the whole space $M^{m,p}_{A}(\mathbb{R}^{d})$.
	\end{proof}

	\section{Modulation space well-posedness}
	
	We show the well-posedness of the generalized magnetic Schr\"odinger equation \eqref{eq:electro-magnetic weyl schrodinger} by constructing its propagator, based on   the parametrix construction in  equation \eqref{eq:parametrix equation} and  Duhamel's principle. For this we need the following definition of a new space and a lemma for the Volterra equation involving the operator $K$ defined in the Proposition \ref{prop:parametrix}.
	
	\begin{defi}\label{defi:t-dependent T}
		Define the time-dependent wavepacket transform $\mathcal{T}^{A}_{\lambda,t}$ on $u \in \mathcal{S}' ([0,T]\times\mathbb{R}^{d})$ by 
		$$
		\tilde{u}(t;x^{t},\xi^{t}) := \mathcal{T}^{A}_{\lambda,t}u(t;x^{t},\xi^{t}) := \int_{\mathbb{R}^{d}} \overline{g^{A,\lambda}_{x^{t},\xi^{t}}}(y)u(t;y) dy.
		$$
		For $1\leq p \leq \infty$  and $m \in \mathbb{R}$, we define the Banach spaces $\mathcal{M}^{m,p}_{A}(\mathbb{R}^{d},L^{1}([0,T]))$ characterized by the mixed $L^{p}$-norm of the transform $\tilde{u}$ on $[0,T]\times \mathbb{R}^{2d}$. That is, we set
\begin{alignat*}{2}
			\|u\|_{\mathcal{M}^{m,p}_{A}(\mathbb{R}^{d},L^{1}([0,T]))} 
			&:=  \left(\int_{\mathbb{R}^{2d}}\left(\nu_{\lambda,m}(x,\xi)\int_{[0,T]}\left|\tilde{u}(t;x^{t},\xi^{t})\right|dt\right)^{p}\ dxd\xi\right)^{1/p}, &\quad 1\leq p < \infty,\\
			\|u\|_{\mathcal{M}^{m,\infty}_{A}(\mathbb{R}^{d},L^{1}([0,T]))} 
			&:= \underset{(x,\xi) \in \mathbb{R}^{2d}}{\operatorname{ess\ sup}}\  \nu_{\lambda,m}(x,\xi) \int_{[0,T]}\left|\tilde{u}(t;x^{t},\xi^{t})\right|dt, &\quad p = \infty,
\end{alignat*}
		and define for $1\leq p\leq \infty$
		$$
		\mathcal{M}^{m,p}_{A}(\mathbb{R}^{d},L^{1}([0,T])) := \left\{u \in \mathcal{S}' ([0,T]\times\mathbb{R}^{d}) \bigg|\  \|u\|_{\mathcal{M}^{m,p}_{A}(\mathbb{R}^{d},L^{1}([0,T]))} < \infty \right\}.
		$$
	\end{defi}
	
	\begin{lem}\label{lem:volterra}
		 Let $T>0$ be  sufficiently small.  For every  $u_{0} \in M^{m,p}_{A}(\mathbb{R}^{d})$ there exists a unique solution $v \in \mathcal{M}^{m,p}_{A}(\mathbb{R}^{d},L^{1}([0,T]))$ to the Volterra equation
		\begin{equation}\label{eq:volterra}
			v(t;y) = -K(t,0)u_{0}(y) -i\int_{0}^{t}K(t,s)v(s;y)\ ds, \qquad  y \in \R^d,t>0.
		\end{equation}
		Moreover, for the solution $v$ we have
		\begin{equation}\label{eq:v bound}
			\|v\|_{\mathcal{M}^{m,p}_{A}(\mathbb{R}^{d},L^{1}([0,T]))} \lesssim C_{T,h,B}\|u_{0}\|_{M^{m,p}_{A}(\mathbb{R}^{d})}.
		\end{equation}
	\end{lem}
	
	\begin{proof}
		 In view of Lemma \ref{lem:weighted kernel bound} and using arguments regarding the weights $\nu_{\lambda,m}$ similar to the ones given in the proof of Proposition \ref{prop:parametrix}, it suffices to prove  the unweighted version with $m=0$ as long as the relevant phase space kernel has sufficient off-diagonal decay.  In the following we show  the unweighted case with the spaces $M^{p}_{A}(\mathbb{R}^{d})$ and $\mathcal{M}^{p}_{A}(\mathbb{R}^{d},L^{1}([0,T])) := \mathcal{M}^{0,p}_{A}(\mathbb{R}^{d},L^{1}([0,T]))$, 
		   from which the   desired off-diagonal decay of the kernels  will become apparent.  The necessary modifications for the weighted case are outlined at the end of the proof.

		By the Banach fixed point theorem, to show the unique existence of  a  solution to (\ref{eq:volterra}), it suffices to show that the mapping $\Psi_{u_{0}}$ defined by
		\begin{equation}\label{eq:volterra solution map}
			\Psi_{u_{0}}(v)(y) := -K(t,0)u_{0}(y) -i\int_{0}^{t}K(t,s)v(s;y)\ ds,\quad v \in \mathcal{M}^{p}_{A}(\mathbb{R}^{d},L^{1}([0,T])),
		\end{equation}
		is a contraction mapping on $\mathcal{M}^{p}_{A}(\mathbb{R}^{d},L^{1}([0,T]))$.
		
		We first show that $K(t,0)u_{0}$ is in the space $\mathcal{M}^{p}_{A}(\mathbb{R}^{d},L^{1}([0,T]))$, by showing $K(t,0)$ is a bounded linear operator between $M^{p}_{A}(\mathbb{R}^{d})$ and $\mathcal{M}^{p}_{A}(\mathbb{R}^{d},L^{1}([0,T]))$, that is
		\begin{equation}\label{eq:K bound}
			\|K(\cdot,0)u_{0}\|_{\mathcal{M}^{p}_{A}(\mathbb{R}^{d},L^{1}([0,T]))} \leq C_{K}\|u_{0}\|_{M^{p}_{A}(\mathbb{R}^{d})},\quad  u_{0} \in M^{p}_{A}(\mathbb{R}^{d}).
		\end{equation}
		The above bound can be obtain by interpolating the following bounds (cf. \cite[Section 1.18.7, Theorem 1]{Triebel_Interpolation})
		\begin{equation}\label{eq:K bounds}
			\begin{gathered}
				\|K(\cdot,0)u_{0}\|_{\mathcal{M}^{1}_{A}(\mathbb{R}^{d},L^{1}([0,T]))} \leq C_{K}'\|u_{0}\|_{M^{1}_{A}(\mathbb{R}^{d})},\\
				\|K(\cdot,0)u_{0}\|_{\mathcal{M}^{\infty}_{A}(\mathbb{R}^{d},L^{1}([0,T]))} \leq C_{K}''\|u_{0}\|_{M^{\infty}_{A}(\mathbb{R}^{d})}.
			\end{gathered}
		\end{equation}
		
		We begin with the bound for $K(t,0): M^{1}_{A}(\mathbb{R}^{d}) \to \mathcal{M}^{1}_{A}(\mathbb{R}^{d},L^{1}([0,T]))$. By definition of $K$ in \eqref{def:K}, we have
		\begin{equation}\label{eq:K wavepacket transform}
			\begin{aligned}
				& \mathcal{T}^{A}_{\lambda,t}K(t,0)u_{0}(z^{t},\zeta^{t})\\
				& = \int_{\mathbb{R}^{d}} \overline{g^{A,\lambda}_{z^{t},\zeta^{t}}}(y) \left(\int_{\mathbb{R}^{2d}} (\mathcal{R}_{1}+\mathcal{R}_{2})g^{A,\lambda}_{x^{t},\xi^{t}}(y) e^{i\psi(t)}\tilde{u}_{0}(x,\xi)\ dxd\xi\right)\ dy\\
				& = \int_{\mathbb{R}^{2d}} \left(\int_{\mathbb{R}^{d}} \overline{g^{A,\lambda}_{z^{t},\zeta^{t}}}(y) (\mathcal{R}_{1}+\mathcal{R}_{2})g^{A,\lambda}_{x^{t},\xi^{t}}(y)\ dy\right) e^{i\psi(t)}\tilde{u}_{0}(x,\xi)\ dxd\xi.
			\end{aligned}
		\end{equation}
		Denote the integral kernel above as $G_{K}$, with
		\begin{equation}\label{eq:G}
			G_{K}(t;x,\xi,z,\zeta) := \int_{\mathbb{R}^{d}} \overline{g^{A,\lambda}_{z^{t},\zeta^{t}}}(y)(\mathcal{R}_{1}+\mathcal{R}_{2})g^{A,\lambda}_{x^{t},\xi^{t}}(y)\ dy.
		\end{equation}
		We claim that $G_{K}$ satisfies the bound
		\begin{equation}\label{eq:G bound}
			\begin{aligned}
				& |G_{K}(t;x,\xi,z,\zeta)|\\
				& \leq C_{N,h,B}(\lambda^{C_{N}}+\lambda^{-2}\langle x^{t}\rangle^{-1-\epsilon}|\xi^{t}|) \langle \lambda(z^{t}-x^{t})\rangle^{-N}\langle \lambda^{-1}(\zeta^{t}-\xi^{t})\rangle^{-N},\quad  \forall N \in \mathbb{N},
			\end{aligned}
		\end{equation}
		which we will prove in the next section as part of Lemma \ref{lem:G bound}. From this, we obtain 
		$$
		\begin{aligned}
			& \|K(t,0)u_{0}\|_{\mathcal{M}^{1}_{A}(\mathbb{R}^{d},L^{1}([0,T]))}\\
			& = \int_{\mathbb{R}^{2d}} \int_{[0,T]}\left|\mathcal{T}^{A}_{\lambda,t}K(t,0)u_{0}(z^{t},\zeta^{t})\right|dt\ dzd\zeta\\
			& \lesssim \int_{\mathbb{R}^{2d}} \int_{[0,T]} \bigg[\left(\int_{\mathbb{R}^{2d}} \langle \lambda(z^{t}-x^{t})\rangle^{-d-1}\langle \lambda^{-1}(\zeta^{t}-\xi^{t})\rangle^{-d-1}\ dz^{t}d\zeta^{t}\right)\times\\
			& \quad\quad \times C_{d,h,B}(\lambda^{C_{d}}+\lambda^{-2}\langle x^{t}\rangle^{-1-\epsilon}|\xi^{t}|) |\tilde{u}_{0}(x,\xi)|\bigg]\ dtdxd\xi\\
			& \lesssim \int_{\mathbb{R}^{2d}} \left(\int_{[0,T]} C_{d,h,B}(\lambda^{C_{d}}+\lambda^{-2}\langle x^{t}\rangle^{-1-\epsilon}|\xi^{t}|)\ dt\right) |\tilde{u}_{0}(x,\xi)|\ dxd\xi\\
			&  \lesssim C_{d,h,B}(T\lambda^{C_{d}}+\lambda^{-2}C_{\epsilon}(1+T)) \int_{\mathbb{R}^{2d}} |\tilde{u}_{0}(x,\xi)|\ dxd\xi\\
			& \lesssim C_{d,h,B,\epsilon}(T\lambda^{C_{d}}+\lambda^{-2}) \|u_{0}\|_{M^{1}_{A}(\mathbb{R}^{d})},
		\end{aligned}
		$$
		where for the $t$-integral we used Lemma \ref{lem:t-averaging}.
		
		For the bound $K(t,0): M^{\infty}_{A}(\mathbb{R}^{d}) \to \mathcal{M}^{\infty}_{A}(\mathbb{R}^{d},L^{1}([0,T]))$, we proceed similarly as above, except that we use the other bound for $G_{K}$ from Lemma \ref{lem:G bound}. That is,
		\begin{equation}\label{eq:G bound alternative}
			\begin{aligned}
				& |G_{K}(t;x,\xi,z,\zeta)|\\
				& \lesssim C_{N,h,B}(\lambda^{C_{N}}+\lambda^{-2}\langle z^{t}\rangle^{-1-\epsilon}|\zeta^{t}|) \langle \lambda(z^{t}-x^{t})\rangle^{-N}\langle \lambda^{-1}(\zeta^{t}-\xi^{t})\rangle^{-N},\quad  \forall N \in \mathbb{N}.
			\end{aligned}
		\end{equation}
		With the above bound and Lemma \ref{lem:flow}, we obtain
		$$
		\begin{aligned}
			& \|K(t,0)u_{0}\|_{\mathcal{M}^{\infty}_{A}(\mathbb{R}^{d},L^{1}([0,T]))}\\
			& =  \underset{(z,\zeta) \in \mathbb{R}^{2d}}{\operatorname{ess\ sup}} \int_{[0,T]}\left|\mathcal{T}^{A}_{\lambda,t}K(t,0)u_{0}(z^{t},\zeta^{t})\right|dt\\
			& \lesssim \underset{(z,\zeta) \in \mathbb{R}^{2d}}{\operatorname{ess\ sup}}\  \int_{[0,T]} \bigg[C_{d,h,B}(\lambda^{C_{d}}+\lambda^{-2}\langle z^{t}\rangle^{-1-\epsilon}|\zeta^{t}|)\times\\
			&\quad\quad \left.\times \int_{\mathbb{R}^{2d}} \langle \lambda(z^{t}-x^{t})\rangle^{-d-1}\langle \lambda^{-1}(\zeta^{t}-\xi^{t})\rangle^{-d-1} |\tilde{u}_{0}(x,\xi)|\ dxd\xi\right] dt\\
			& \lesssim \underset{(z,\zeta) \in \mathbb{R}^{2d}}{\operatorname{ess\ sup}}\ \int_{[0,T]} C_{d,h,B}(\lambda^{C_{d}}+\lambda^{-2}\langle z^{t}\rangle^{-1-\epsilon}|\zeta^{t}|)\ dt \times\\
			&\quad\quad \times \underset{t \in [0,T]}{\operatorname{ess\ sup}} \left( \underset{(x,\xi) \in \mathbb{R}^{2d}}{\operatorname{ess\ sup}}|\tilde{u}_{0}(x,\xi)|\int_{\mathbb{R}^{2d}} \langle \lambda(z^{t}-x^{t})\rangle^{-d-1}\langle \lambda^{-1}(\zeta^{t}-\xi^{t})\rangle^{-d-1}\ dx^{t}d\xi^{t}\right).
		\end{aligned}
		$$
		We then again use Lemma \ref{lem:t-averaging} for the $t$-integral and notice  that 
		$$
		\int_{\mathbb{R}^{2d}} \langle \lambda(z^{t}-x^{t})\rangle^{-d-1}\langle \lambda^{-1}(\zeta^{t}-\xi^{t})\rangle^{-d-1}\ dx^{t}d\xi^{t} \leq C_{d}
		$$
		uniformly in $t$,  which yields 
		$$
		\|K(t,0)u_{0}\|_{\mathcal{M}^{\infty}_{A}(\mathbb{R}^{d},L^{1}([0,T]))} \lesssim C_{d,h,B,\epsilon}(T\lambda^{C_{d}}+\lambda^{-2})\|u\|_{M^{\infty}_{A}(\mathbb{R}^{d})}.
		$$
		
		With $K(t,0)u_{0} \in \mathcal{M}^{p}_{A}(\mathbb{R}^{d},L^{1}([0,T]))$, to show $\Psi_{u_{0}}$ defined in (\ref{eq:volterra solution map}) is a contraction mapping on $\mathcal{M}^{p}_{A}(\mathbb{R}^{d},L^{1}([0,T]))$, it suffices to show the map $F: [0,T]\times \mathbb{R}^{d} \to [0,T]\times \mathbb{R}^{d}$ defined as
		$$
			F(v)(t;y) := -\int_{0}^{t} K(t,s)v(s;y)\ ds
		$$
		is a contraction mapping on $\mathcal{M}^{p}_{A}(\mathbb{R}^{d},L^{1}([0,T]))$. Similar to the bound (\ref{eq:K bound}) of $K(t,0)$, the contraction bound
		$$
			\|F(v)\|_{\mathcal{M}^{p}_{A}(\mathbb{R}^{d},L^{1}([0,T]))} \leq C_{F}\|v\|_{\mathcal{M}^{p}_{A}(\mathbb{R}^{d},L^{1}([0,T]))},\quad C_{F} \in (0,1)
		$$
		can be obtain by interpolating the bounds
		\begin{equation}\label{eq:F bounds}
			\begin{gathered}
				\|F(v)\|_{\mathcal{M}^{1}_{A}(\mathbb{R}^{d},L^{1}([0,T]))} \leq C_{F}'\|v\|_{\mathcal{M}^{1}_{A}(\mathbb{R}^{d},L^{1}([0,T]))},\quad C_{F}' \in (0,1),\\
				\|F(v)\|_{\mathcal{M}^{\infty}_{A}(\mathbb{R}^{d},L^{1}([0,T]))} \leq C_{F}''\|v\|_{\mathcal{M}^{\infty}_{A}(\mathbb{R}^{d},L^{1}([0,T]))},\quad C_{F}'' \in (0,1).
			\end{gathered}
		\end{equation}
		Notice that for the relevant wavepacket transform $\mathcal{T}^{A}_{\lambda,t}F(v)(t;z^{t},\zeta^{t})$, we can re-write again with the integral kernel $G_{K}$ defined in (\ref{eq:G})
		$$
			\begin{aligned}
				\mathcal{T}^{A}_{\lambda,t}F(v)(t;z^{t},\zeta^{t}) & = \int_{\mathbb{R}^{d}} \overline{g^{A,\lambda}_{z^{t},\zeta^{t}}}(y)\int_{0}^{t}\left(\int_{\mathbb{R}^{2d}} (\mathcal{R}_{1}+\mathcal{R}_{2})g^{A,\lambda}_{x^{t},\xi^{t}}(y)\left(e^{i\psi(t)-i\psi(s)}\tilde{v}(s;x^{s},\xi^{s})\right)\ dxd\xi\right)\ dsdy\\
				& = \int_{\mathbb{R}^{2d}}\left(\int_{\mathbb{R}^{d}} \overline{g^{A,\lambda}_{z^{t},\zeta^{t}}}(y)(\mathcal{R}_{1}+\mathcal{R}_{2})g^{A,\lambda}_{x^{t},\xi^{t}}(y)\ dy\right) \left(\int_{0}^{t} e^{i\psi(t)-i\psi(s)}\tilde{v}(s;x^{s},\xi^{s})\ ds\right)\ dx d\xi\\
				& = \int_{\mathbb{R}^{2d}} G_{K}(t;x,\xi,z,\zeta) \left(\int_{0}^{t} e^{i\psi(t)-i\psi(s)}\tilde{v}(s;x^{s},\xi^{s})\ ds\right)\ dx d\xi.
			\end{aligned}
		$$
		The above representation is quite similar to (\ref{eq:K wavepacket transform}), where the function $e^{i\psi(t)}\tilde{u}_{0}(x,\xi)$ is replaced by the slightly more complicated $\int_{0}^{t} e^{i\psi(t)-i\psi(s)}\tilde{v}(s;x^{s},\xi^{s})\ ds$. Therefore, the proof of the  bounds (\ref{eq:K bounds})  for $K$  can  easily be modified to prove the  bounds (\ref{eq:F bounds}) for $F$, which again relies on the kernel bounds for $G_{K}$ in Lemma \ref{lem:G bound}.

		For the contraction bound on $\mathcal{M}^{1}_{A}(\mathbb{R}^{d},L^{1}([0,T]))$, we see that
		$$
		\begin{aligned}
			& \|F(v)\|_{\mathcal{M}^{1}_{A}(\mathbb{R}^{d},L^{1}([0,T]))}\\
			&\quad \leq \int_{\mathbb{R}^{2d}} \int_{[0,T]} \int_{\mathbb{R}^{2d}} |G_{K}(t;x,\xi,z,\zeta)| \left(\int_{0}^{t}  |\tilde{v}(s;x^{s},\xi^{s})| ds\right) dx d\xi\ dt\ dzd\zeta\\
			&\quad \lesssim \int_{\mathbb{R}^{2d}} \int_{[0,T]} \bigg[\left(\int_{\mathbb{R}^{2d}} \langle \lambda(z^{t}-x^{t})\rangle^{-d-1}\langle \lambda^{-1}(\zeta^{t}-\xi^{t})\rangle^{-d-1}\ dz^{t}d\zeta^{t}\right)\times\\
			&\quad \quad\quad \times C_{d,h,B}(\lambda^{C_{d}}+\lambda^{-2}\langle x^{t}\rangle^{-1-\epsilon}|\xi^{t}|) \left(\int_{0}^{t}  |\tilde{v}(s;x^{s},\xi^{s})| ds\right)\bigg]\ dtdxd\xi\\
			&\quad \leq \int_{\mathbb{R}^{2d}} \left(\int_{[0,T]} C_{d,h,B}(\lambda^{C_{d}}+\lambda^{-2}\langle x^{t}\rangle^{-1-\epsilon}|\xi^{t}|)\ dt\right) \underset{t \in [0,T]}{\operatorname{ess\ sup}} \left(\int_{0}^{t}  |\tilde{v}(s;x^{s},\xi^{s})| ds\right)\ dxd\xi\\
			&\quad \lesssim C_{d,h,B}(T\lambda^{C_{d}}+\lambda^{-2}C_{\epsilon}(1+T)) \int_{\mathbb{R}^{2d}} \left(\int_{0}^{T}  |\tilde{v}(s;x^{s},\xi^{s})| ds\right)\ dxd\xi\\
			&\quad \leq C_{d,h,B,\epsilon}(T\lambda^{C_{d}}+\lambda^{-2}) \|v\|_{\mathcal{M}^{1}_{A}(\mathbb{R}^{d},L^{1}([0,T]))}.
		\end{aligned}
		$$
		For the contraction bound on  $\mathcal{M}^{\infty}_{A}(\mathbb{R}^{d},L^{1}([0,T]))$, we use
		$$
		\begin{aligned}
			& \|F(v)\|_{\mathcal{M}^{\infty}_{A}(\mathbb{R}^{d},L^{1}([0,T]))}\\
			&\quad \leq \esssup_{(z,\zeta) \in \mathbb{R}^{2d}}  \int_{[0,T]} \int_{\mathbb{R}^{2d}} |G_{K}(t;x,\xi,z,\zeta)| \left(\int_{0}^{t}  |\tilde{v}(s;x^{s},\xi^{s})| ds\right) dx d\xi dt\\
			&\quad \lesssim \underset{(z,\zeta) \in \mathbb{R}^{2d}}{\operatorname{ess\ sup}} \bigg[\int_{[0,T]} C_{d,h,B}(\lambda^{C_{d}}+\lambda^{-2}\langle z^{t}\rangle^{-1-\epsilon}|\zeta^{t}|)\ dt \times\\
			&\quad \quad\quad \times \underset{t \in [0,T]}{\operatorname{ess\ sup}} \int_{\mathbb{R}^{2d}} \langle \lambda(z^{t}-x^{t})\rangle^{-d-1}\langle \lambda^{-1}(\zeta^{t}-\xi^{t})\rangle^{-d-1} \left(\int_{0}^{t} |\tilde{v}(s;x^{s},\xi^{s})|\ ds\right) dxd\xi \bigg] \\
			&\quad \leq  C_{d,h,B}(T\lambda^{C_{d}}+\lambda^{-2}C_{\epsilon})\  \underset{(z,\zeta) \in \mathbb{R}^{2d}}{\operatorname{ess\ sup}} \; \underset{t \in [0,T]}{\operatorname{ess\ sup}} \left(\int_{\mathbb{R}^{2d}} \langle \lambda(z^{t}-x^{t})\rangle^{-d-1}\langle \lambda^{-1}(\zeta^{t}-\xi^{t})\rangle^{-d-1}\ dxd\xi \times\right.\\
			&\quad \quad\quad \times \left. \underset{(x,\xi) \in \mathbb{R}^{2d}}{\operatorname{ess\ sup}} \left(\int_{0}^{t} |\tilde{v}(s;x^{s},\xi^{s})|\ ds\right) \right) \\
			&\quad \leq C_{d,h,B}(T\lambda^{C_{d}}+\lambda^{-2}C_{\epsilon}(1+T))\  \underset{(x,\xi) \in \mathbb{R}^{2d}}{\operatorname{ess\ sup}}\ \;  \underset{t \in [0,T]}{\operatorname{ess\ sup}} \int_{0}^{t} |\tilde{v}(s;x^{s},\xi^{s})|\ ds\\
			&\quad \leq C_{d,h,B,\epsilon}(T\lambda^{C_{d}}+\lambda^{-2}) \|v\|_{\mathcal{M}^{\infty}_{A}(\mathbb{R}^{d},L^{1}([0,T]))}.
		\end{aligned}
		$$
		For given constants $C_{d,h,B}, C_{d}, C_{\epsilon}$, we can choose sufficiently large $\lambda$ and small $T$ such that
		$$
		C^{*} := C_{d,h,B,\epsilon}(T\lambda^{C_{d}}+\lambda^{-2}) < 1,
		$$
		then by complex interpolation, for $1\leq p\leq \infty$ the operator $F$ is indeed a contraction on the Banach space $\mathcal{M}^{p}_{A}(\mathbb{R}^{d},L^{1}([0,T]))$, and we have the existence and uniqueness of $v$ as the solution to (\ref{eq:volterra}) by the Banach fixed point theorem.
		
		Moreover, using  equation (\ref{eq:volterra}) we have for the solution $v$
		$$
		\begin{aligned}
			\|v\|_{\mathcal{M}^{p}_{A}(\mathbb{R}^{d},L^{1}([0,T]))} & \leq \|K(t,0)u_{0}\|_{\mathcal{M}^{p}_{A}(\mathbb{R}^{d},L^{1}([0,T]))} + \|F(v)\|_{\mathcal{M}^{p}_{A}(\mathbb{R}^{d},L^{1}([0,T]))}\\
			& \leq C_{K}\|u_{0}\|_{M^{p}_{A}(\mathbb{R}^{d})} + C_{F}\|v\|_{\mathcal{M}^{p}_{A}(\mathbb{R}^{d},L^{1}([0,T]))}.
		\end{aligned}
		$$
		Since $C_{F} \in (0,1)$, we obtain
		$$
		\|v\|_{\mathcal{M}^{p}_{A}(\mathbb{R}^{d},L^{1}([0,T]))} \leq C_{K}(1-C_{F})^{-1}\|u_{0}\|_{M^{p}_{A}(\mathbb{R}^{d})}.
		$$
		
		 To derive the weighted version, it suffices to notice that for the weight $\nu_{\lambda,m}$ we can use the flow property \eqref{eq:Gronwall} and \eqref{eq:Gronwall backward} to obtain
		$$
		\begin{aligned}
			\nu_{\lambda,m}(z,\zeta) & \lesssim C_{T,h,B}^{m} \nu_{\lambda,m}(z^{t},\zeta^{t})\nu_{\lambda,-m}(x,\xi)\nu_{\lambda,m}(x,\xi)\\
			& \lesssim \langle \lambda^{-1}(z^{t}-x^{t})\rangle^{m} \langle \lambda^{-1}(\zeta^{t}-\xi^{t})\rangle^{m} \nu_{\lambda,m}(x^{t},\xi^{t})\nu_{\lambda,-m}(x,\xi)\nu_{\lambda,m}(x,\xi)\\
			& \lesssim C_{T,h,B}^{2m} \langle \lambda^{-1}(z^{t}-x^{t})\rangle^{m} \langle \lambda^{-1}(\zeta^{t}-\xi^{t})\rangle^{m} \nu_{\lambda,m}(x,\xi).
		\end{aligned}
		$$
		Then we can modify the arguments for the unweighted version in the obvious way with the above bound, and combine them with the sufficient off-diagonal decay of $G_{K}$ as in \eqref{eq:G bound} and \eqref{eq:G bound alternative} to obtain the desired weighted version involving $M^{m,p}_{A}(\mathbb{R}^{d})$ and $\mathcal{M}^{m,p}_{A}(\mathbb{R}^{d},L^{1}([0,T]))$.
	\end{proof}
	
	We are now ready to prove Theorem \ref{thm:main}.
	
	\begin{proof}[Proof of Theorem \ref{thm:main}]
		 We define the propagator $S$ for $u_{0} \in \mathcal{S}(\mathbb{R}^{d})$ by
		$$
			S(t)u_{0} := \tilde{S}(t,0)u_{0} + i\int_{0}^{t}\tilde{S}(t,s)v(s)\ ds, \quad 0\leq t\leq T,
		$$
		where $v$ denotes the unique solution  to the Volterra equation \eqref{eq:volterra}  in Lemma \ref{lem:volterra} and  $\tilde S$ is defined in \eqref{def:S tilde}. Then  there holds   $S(0) = \operatorname{Id}$, and $u(t) = S(t)u_{0}$ solves the equation (\ref{eq:electro-magnetic weyl schrodinger}),  since by   direct computation we have  
		$$
		(D_{t}+\operatorname{Op}^{A}(h))S(t)u_{0}(y) = 0.
		$$
		
		Next we show for $1 \leq p \leq \infty, m \geq 0$ and $t \in [0,T]$, the propagator $S(t)$ is bounded on  $M^{m,p}_{A}(\mathbb{R}^{d})$. Fix $m,p$ and $t$.  By  definition of  $S(t)$, we have
		\begin{equation}\label{eq:S bound}
			\|S(t)u_{0}\|_{M^{m,p}_{A}(\mathbb{R}^{d})} \leq \|\tilde{S}(t,0)u_{0}\|_{M^{m,p}_{A}(\mathbb{R}^{d})} + \left\|i\int_{0}^{t}\tilde{S}(t,s)v(s,\cdot)\ ds\right\|_{M^{m,p}_{A}(\mathbb{R}^{d})}.
		\end{equation}
		For the terms on the right hand side above, we have by Proposition \ref{prop:parametrix} 
		\begin{equation}\label{eq:parametrix bound}
			\|\tilde{S}(t,s)u_{0}\|_{M^{m,p}_{A}(\mathbb{R}^{d})} \lesssim C_{T,h,B,m} \|u_{0}\|_{M^{m,p}_{A}(\mathbb{R}^{d})},
		\end{equation}
		so it remains to check the boundedness of the term involving $v$. Here, we have
		$$
		\left\|i\int_{0}^{t}\tilde{S}(t,s)v(s;\cdot)\ ds\right\|^{p}_{M^{m,p}_{A}(\mathbb{R}^{d})} = \int_{\mathbb{R}^{2d}} \left| \nu_{\lambda,m}(z,\zeta)\int_{\mathbb{R}^{d}} \overline{g^{A,\lambda}_{z,\zeta}}(y) \left(\int_{0}^{t}\tilde{S}(t,s)v(s;y)\ ds\right)\ dy \right|^{p} dz d\zeta,
		$$
		which depends on the $L^{p}$-integrability of the phase space function $\tilde{v}'(t;z,\zeta)$, defined by
		$$
		\tilde{v}'(t;z,\zeta) := \nu_{\lambda,m}(z,\zeta)\int_{\mathbb{R}^{d}} \overline{g^{A,\lambda}_{z,\zeta}}(y) \left(\int_{0}^{t}\tilde{S}(t,s)v(s;y)\ ds\right )\ dy.
		$$
		Writing out the action of $\tilde{S}$
		$$
		\int_{0}^{t}\tilde{S}(t,s)v(s;y)\ ds = \int_{0}^{t}\left(\int_{\mathbb{R}^{2d}}e^{i\psi(t)-i\psi(s)}g^{A,\lambda}_{x^{t},\xi^{t}}(y)\tilde{v}(s;x^{s},\xi^{s})\ dx d\xi\right)\ ds,
		$$
		we obtain an expression of $\tilde{v}'(t;z,\zeta)$ with the integral kernel $G^{A,\lambda}$ as defined in (\ref{def:G^A,lambda})
		$$
		\begin{aligned}
			\tilde{v}'(t;z,\zeta)
			& = \int_{\mathbb{R}^{2d}}\nu_{\lambda,m}(z,\zeta)\left( \int_{\mathbb{R}^{d}} \overline{g^{A,\lambda}_{z,\zeta}}(y)g^{A,\lambda}_{x^{t},\xi^{t}}(y)\ dy \right)\left(\int_{0}^{t} e^{i\psi(t)-i\psi(s)}\tilde{v}(s;x^{s},\xi^{s})\ ds\right)\ dx d\xi\\
			& = \int_{\mathbb{R}^{2d}} \nu_{\lambda,m}(z,\zeta)G^{A,\lambda}(z,\zeta,x^{t},\xi^{t})\nu_{\lambda,-m}(x,\xi) \left(\nu_{\lambda,m}(x,\xi)\int_{0}^{t} e^{i\psi(t)-i\psi(s)}\tilde{v}(s;x^{s},\xi^{s})\ ds\right)\ dx d\xi.
		\end{aligned}
		$$
		Since we have proved that $G^{A,\lambda}$ has off-diagonal decays as in (\ref{eq:G^A,lambda bound}), together with
		$$
		\nu_{\lambda,m}(z,\zeta)\nu_{\lambda,-m}(x,\xi) \lesssim C_{T,h,B}^{m}\langle\lambda^{-1}(z-x^{t}) \rangle^{m}\langle \lambda^{-1}(\zeta-\xi^{t})\rangle^{m},
		$$
		we can apply Young's inequality for integral operators to obtain
 	$$
		\left\|i\int_{0}^{t}\tilde{S}(t,s)v(s;\cdot)\ ds\right\|_{M^{m,p}_{A}(\mathbb{R}^{d})} = \|\tilde{v}'(t)\|_{L^{p}(\mathbb{R}^{2d})} \lesssim C_{T,h,B}^{m} \left\|\nu_{\lambda,m}(x,\xi)\int_{0}^{t} e^{i\psi(t)-i\psi(s)}\tilde{v}(s;x^{s},\xi^{s})\ ds\right\|_{L^{p}(\mathbb{R}^{2d})}.
		$$
		The right hand side above is in fact uniformly bound for $t \in [0,T]$ by
		$$
			\begin{aligned}
				& \left(\int_{\mathbb{R}^{2d}} \left|\nu_{\lambda,m}(x,\xi)\int_{0}^{t} e^{i\psi(t)-i\psi(s)}\tilde{v}(s;x^{s},\xi^{s})\ ds\right|^{p} dxd\xi\right)^{1/p}\\
				\leq &\  \left(\int_{\mathbb{R}^{2d}} \left|\nu_{\lambda,m}(x,\xi)\int_{0}^{T} |\tilde{v}(s;x^{s},\xi^{s})|\ ds\right|^{p} dxd\xi\right)^{1/p}\\
				= &\  \|v\|_{\mathcal{M}^{m,p}_{A}(\mathbb{R}^{d},L^{1}([0,T]))}.
			\end{aligned}
		$$
		Together with (\ref{eq:v bound}) in Lemma \ref{lem:volterra}, we now have
		\begin{equation}\label{eq:v' tilde bound}
			\left\|i\int_{0}^{t}\tilde{S}(t,s)v(s;\cdot)\ ds\right\|^{p}_{M^{m,p}_{A}(\mathbb{R}^{d})} \lesssim C_{T,h,B}^{m} \|v\|_{\mathcal{M}^{m,p}_{A}(\mathbb{R}^{d},L^{1}([0,T]))} \lesssim C_{T,h,B,m}\|u_{0}\|_{M^{m,p}_{A}(\mathbb{R}^{d})}.
		\end{equation}
		Plugging (\ref{eq:parametrix bound}) and (\ref{eq:v' tilde bound}) into (\ref{eq:S bound}), we obtain the desired bound for $S(t)$
		$$
			\|S(t)u_{0}\|_{M^{m,p}_{A}(\mathbb{R}^{d})} \lesssim C_{T,h,B,m}\|u_{0}\|_{M^{m,p}_{A}(\mathbb{R}^{d})}.
		$$
		This estimate implies the uniqueness of the solution $u(t) = S(t)u_{0}$, and also rules out the possibility of finite-time blow-up, so we can extend the solution globally in $t$.
		
		In the case of $u_{0} \in M^{2,p}_{A}(\mathbb{R}^{d})$, the above bound implies $u(t) = S(t)u_{0} \in M^{2,p}_{A}(\mathbb{R}^{d}) \subset M^{p}_{A}(\mathbb{R}^{d})$ for all $t \in [0,T]$. To see $u \in W^{1,1}([0,T],M^{p}_{A}(\mathbb{R}^{d}))$, it suffice to note
		$$
		\|D_{t}u(t)\|_{M^{p}_{A}(\mathbb{R}^{d})} = \|\operatorname{Op}^{A}(h)u(t)\|_{M^{p}_{A}(\mathbb{R}^{d})} \lesssim C_{h} \|u(t)\|_{M^{2,p}_{A}(\mathbb{R}^{d})} \lesssim C_{T,h,B}  \|u_{0}\|_{M^{2,p}_{A}(\mathbb{R}^{d})},
		$$
		where the first inequality follows from the boundedness of $\operatorname{Op}^{A}(h)$ given by Lemma \ref{lem:Op^A boundedness}.
	\end{proof}

	\subsection{Integral kernel estimates}
	
	In this subsection we prove the bounds (\ref{eq:G bound}) and (\ref{eq:G bound alternative}) for the integral kernel $G_{K}$.
	
	\begin{lem}\label{lem:G bound}
		The kernel $G_{K}$ defined in (\ref{eq:G})
		$$
		G_{K}(t;x,\xi,z,\zeta) := \int_{\mathbb{R}^{d}} \overline{g^{A,\lambda}_{z^{t},\zeta^{t}}}(y) (\mathcal{R}_{1}+\mathcal{R}_{2})g^{A,\lambda}_{x^{t},\xi^{t}}(y)\ dy,\quad  t \in \mathbb{R}, (x,\xi,z,\zeta) \in \mathbb{R}^{4d},
		$$
		can be bounded by, for all $N \in \mathbb{N}$,
		$$
		\begin{aligned}
			& |G_{K}(t;x,\xi,z,\zeta)|\\
			& \quad \lesssim C_{N,h,B}(\lambda^{C_{N}}+\lambda^{-2}\langle x^{t}\rangle^{-1-\epsilon}|\xi^{t}|) \langle \lambda(z^{t}-x^{t})\rangle^{-N}\langle \lambda^{-1}(\zeta^{t}-\xi^{t})\rangle^{-N},
		\end{aligned}
		$$
		or alternatively
		$$
		\begin{aligned}
			& |G_{K}(t;x,\xi,z,\zeta)|\\
			& \quad \lesssim  C'_{N,h,B}(\lambda^{C'_{N}}+\lambda^{-2}\langle z^{t}\rangle^{-1-\epsilon}|\zeta^{t}|) \langle \lambda(z^{t}-x^{t})\rangle^{-N}\langle \lambda^{-1}(\zeta^{t}-\xi^{t})\rangle^{-N}.
		\end{aligned}
		$$
	\end{lem}
	
	\begin{proof}
		We first note that it suffices to prove the first $G_{K}$ bound above, since the second bound can be derived from the first one by dominating the factor $\langle x^{t}\rangle^{-1-\epsilon}$ and $|\xi^{t}|$ with
		$$
		\begin{aligned}
			\langle x^{t}\rangle^{-1-\epsilon} & = \langle x^{t} - z^{t} + z^{t}\rangle^{-1-\epsilon}\\
			& \lesssim \langle z^{t}\rangle^{-1-\epsilon}\langle x^{t} - z^{t}\rangle^{1+\epsilon}\\
			& \leq \langle z^{t}\rangle^{-1-\epsilon}\langle \lambda(x^{t} - z^{t})\rangle^{2},
		\end{aligned}
		$$
		and
		$$
		\begin{aligned}
			|\xi^{t}| & = |\xi^{t}-\zeta^{t}+\zeta^{t}|\\
			& \leq |\xi^{t}-\zeta^{t}| + |\zeta^{t}|\\
			& \leq \lambda\langle \lambda^{-1}(\xi^{t}-\zeta^{t})\rangle + |\zeta^{t}|.
		\end{aligned}
		$$
		
		By linearity, we can split $G_{K} = G_{\mathcal{R}_{1}} + G_{\mathcal{R}_{2}}$, where
		$$
			G_{\mathcal{R}_{k}}(t;x,\xi,z,\zeta) := \int_{\mathbb{R}^{d}} \overline{g^{A,\lambda}_{z^{t},\zeta^{t}}}(y) \mathcal{R}_{k}g^{A,\lambda}_{x^{t},\xi^{t}}(y)\ dy,\quad k \in \{1,2\}.
		$$
		We start with $G_{1}$, involving the $\mathcal{R}_{1}$ factor as defined in (\ref{def:R_1})
		$$
		\mathcal{R}_{1} = \partial_{\eta}h(t;x^{t},\xi^{t})\cdot \left( - r^{(1)}_{x^{t}}(A(y,x^{t})) + r^{(1)}_{x^{t}}(A(x^{t},y))\right).
		$$
		 Similar to \eqref{eq:G^A,lambda re-write}, $G_{\mathcal{R}_{1}}$ can be re-written with  $W_{B,\lambda}$ from \eqref{def:wavepackets product W} as 
		\begin{equation}\label{eq:G_1}
			\begin{aligned}
				& G_{\mathcal{R}_{1}}(t;x,\xi,z,\zeta)\\
				&\quad = e^{i\varphi^{A}(z^{t},x^{t})}e^{i\zeta^{t}\cdot(z^{t}-x^{t})} \partial_{\eta}h(t;x^{t},\xi^{t}) \int_{\mathbb{R}^{d}} \bigg[e^{i \lambda^{-1}(\xi^{t}-\zeta^{t}) \cdot \lambda(y-x^{t})} \times\\
				&\quad \quad \quad \times  \left( - r^{(1)}_{x^{t}}(A(y,x^{t})) + r^{(1)}_{x^{t}}(A(x^{t},y))\right) W_{B,\lambda}(y,x^{t},z^{t})\bigg]\ d\lambda y.
			\end{aligned}
		\end{equation}
		
		 Notice that for the $\lambda y$-integral, we can do integration by parts to gain decay factors of $\langle \lambda^{-1}(\xi^{t}-\zeta^{t})\rangle^{-2}$, making use of the identity 
		\begin{equation}\label{eq:IBP lambda y}
			\langle \lambda^{-1}(\xi^{t}-\zeta^{t})\rangle^{-2}(1-\Delta_{\lambda y})e^{i\lambda^{-1}(\xi^{t}-\zeta^{t}) \cdot \lambda y} =  e^{i\lambda^{-1}(\xi^{t}-\zeta^{t}) \cdot \lambda y}. 
		\end{equation}
		In order to perform the integration by parts, we check the $\partial_{\lambda y}$-derivatives of the integrand. 
	 By  definition of $r^{(1)}_{x}(A)$ in \eqref{def:magnetic remainder}, we have 
\begin{align*}
			r^{(1)}_{x^{t}}(A_{j}(y,x^{t})) 
			& = \sum_{m,l=1}^{d} (y_{m}-x^{t}_{m})(y_{l}-x^{t}_{l}) \partial_{y_{m}}\partial_{y_{l}}  A_{j}(x^t+\theta(y-x^t),x^t)  \\
			& = \lambda^{-2}\sum_{m,l=1}^{d} \lambda(y_{m}-x^{t}_{m})\lambda(y_{l}-x^{t}_{l}) \partial_{y_{m}}\partial_{y_{l}}A_{j}(x^t+\theta(y-x^t),x^t).
\end{align*}
		So taking $\partial_{\lambda y}$-derivatives on $r^{(1)}_{x}(A)$ involves taking $\partial_{\lambda y}$-derivatives on $\partial_{y_{m}}\partial_{y_{l}}A_{j}$.  By \eqref{eq:A_j^(2)}, it in turn depends on  derivatives of $B$.  Denote by  $B_{jk}^{(n)}$ the $n$-th order  partial derivatives  of $B_{jk}$, then by the decay assumption \eqref{eq:B decay} for $B_{jk}^{(n)}$ with $n \geq 1$, we have 
		for $0\leq s\leq 1, |\alpha| \geq 0$ and $\lambda \geq 1$
\begin{align*}
			 |\partial_{\lambda y}^{\alpha} B_{jk}^{(n)}(x+s(y-x))|
			\leq C_{\alpha,B} \langle x+s(y-x) \rangle^{-1-\epsilon}
			 \lesssim C_{\alpha,B} \langle x \rangle^{-1-\epsilon} \langle \lambda(y-x) \rangle^{1+\epsilon}.
\end{align*}
	 Plugging this into \eqref{eq:A_j^(2)}, 	 we can bound 
		$$
			|\partial_{\lambda y}^{\alpha}r_{x^{t}}^{(1)}(A(y,x^{t}))| \leq C_{\alpha,B}\lambda^{-2}\langle x^{t} \rangle^{-1-\epsilon}\langle \lambda(y-x^{t})\rangle^{4+\epsilon},\quad |\alpha|\geq 0,
		$$
		 and the same bound holds for $r^{(1)}_{x^{t}}(A(x^{t},y))$. Together with the $\lambda y$-derivative bounds for $W_{B,\lambda}$ in \eqref{eq:W derivatives}, that is for all $N,N',N'' \in \mathbb{N}$
		$$
		\begin{aligned}
			& \left|\partial_{\lambda y}^{\alpha}W_{B,\lambda}(y,x^{t},z^{t})\right|\\
			& \quad \lesssim C_{N,N',N'',\alpha,B} \langle \lambda(z^{t}-x^{t}) \rangle^{-N} \langle\lambda(y-z^{t}) \rangle^{-N'} \langle \lambda(y-x^{t}) \rangle^{-N''},
		\end{aligned}
		$$
		 we can repeatedly do integration by parts relying on \eqref{eq:IBP lambda y} for the $\lambda y$-integral in \eqref{eq:G_1} to gain the desired decays.  So we obtain a bound for $G_{\mathcal{R}_{1}}$ of the form
		\begin{equation}\label{eq:G_1 bound}
			\begin{aligned}
				& |G_{\mathcal{R}_{1}}(t;x,\xi,z,\zeta)|\\
				&\quad \leq |\partial_{\eta}h(t;x^{t},\xi^{t})| C_{N,B} \lambda^{-2} \langle x^{t}\rangle^{-1-\epsilon} \langle \lambda(z^{t}-x^{t}) \rangle^{-N} \langle \lambda^{-1}(\zeta^{t}-\xi^{t}) \rangle^{-N}\times\\
				& \quad\quad \times \int_{\mathbb{R}^{d}} \langle\lambda(y-z^{t}) \rangle^{-d-1} \langle \lambda(y-x^{t}) \rangle^{-d-1}\ d\lambda y,\\
				& \quad \lesssim C_{N,h,B}\lambda^{-2}(1+\langle x^{t}\rangle^{-1-\epsilon}|\xi^{t}|) \langle \lambda(z^{t}-x^{t}) \rangle^{-N} \langle \lambda^{-1}(\zeta^{t}-\xi^{t}) \rangle^{-N},\quad  \forall N \geq 0,
			\end{aligned}
		\end{equation}
		where for the last inequality we used the $t$-uniform boundedness of $\partial_{\eta}h$
		$$
		|\partial_{\eta}h(t;x^{t},\xi^{t})| \leq C_{h}(1+|x^{t}|+|\xi^{t}|).
		$$

		For the kernel $G_{\mathcal{R}_{2}}$, we see that 
		$$
		\begin{aligned}
			G_{\mathcal{R}_{2}}(t;x,\xi,z,\zeta) & = \int_{\mathbb{R}^{d}} \overline{g^{A,\lambda}_{z^{t},\zeta^{t}}}(y) \mathcal{R}_{2}g^{A,\lambda}_{x^{t},\xi^{t}}(y)\ dy\\
			& = \int_{\mathbb{R}^{d}} \overline{g^{A,\lambda}_{z^{t},\zeta^{t}}}(y) \operatorname{Op}_{KN}^{A}\left(r^{(1)}_{x^{t},\xi^{t}}(h) + h_{r}\right)g^{A,\lambda}_{x^{t},\xi^{t}}(y)\ dy\\
			& =: \int_{\mathbb{R}^{d}} \overline{g^{A,\lambda}_{z^{t},\zeta^{t}}}(y) \operatorname{Op}_{KN}^{A}\left(r^{*}(t;y,\eta,x,\xi)\right)g^{A,\lambda}_{x^{t},\xi^{t}}(y)\ dy,
		\end{aligned}
		$$
		which shows that the boundedness of  $G_{\mathcal{R}_{2}}$   relies on the boundedness of the symbol
		$$
		r^{*} := r^{(1)}_{x^{t},\xi^{t}}(h) + h_{r}.
		$$
		The symbol $h_{r}$ is relatively simple, since by  definition (\ref{def:KN symbol}), $h_{r}(t) \in S^{0,(0)}(\mathbb{R}^{2d})$, so its derivatives are fully bounded, i.e., 
		$$
		|\partial_{y}^{\alpha}\partial_{\eta}^{\beta}h_{r}(t;y,\eta)| \leq C_{h,\alpha,\beta}, \quad |\alpha|+|\beta| \geq 0.
		$$
		On the other hand, the symbol $r^{(1)}_{x^{t},\xi^{t}}(h)$ and its derivatives can be bounded by factors of $|y-x^{t}|$ and $|\eta-\xi^{t}|$. By  definition \eqref{def:r1 remainder} for $r^{(1)}_{x^{t},\xi^{t}}(h)$, we see that
		$$
		\begin{aligned}
			|\partial_{y}^{\alpha}\partial_{\eta}^{\beta}r^{(1)}_{x^{t},\xi^{t}}(h)(t;y,\eta)| \lesssim \sum_{0\leq|\gamma_{1}|+|\gamma_{2}|\leq2}C_{h,\alpha,\beta} |y-x^{t}|^{|\gamma_{1}|}|\eta-\xi^{t}|^{|\gamma_{2}|}.
		\end{aligned}
		$$
		So we bound $r^{*}$ and its derivatives by
		\begin{equation}\label{eq:r* bounds}
			|\partial_{y}^{\alpha}\partial_{\eta}^{\beta}r^{*}(t;y,\eta,x^{t},\xi^{t})| \lesssim C_{h,\alpha,\beta}\left(1+ \sum_{0\leq|\gamma_{1}|+|\gamma_{2}|\leq2} |y-x^{t}|^{|\gamma_{1}|}|\eta-\xi^{t}|^{|\gamma_{2}|}\right).
		\end{equation}
		In the following, the variable $t$ does not play a role, so we suppress it from the notation. We first take a closer look at the pseudo-differentiation
		$$
			\begin{aligned}
			 & \operatorname{Op}^{A}_{KN}\left(r^{*}\right)g^{A,\lambda}_{x,\xi}(y)\\
				&\quad = (2\pi)^{-d} \int_{\mathbb{R}^{2d}} e^{i(y-y')\cdot\eta} e^{-i\varphi^{A}(y',y)}r^{*}(y,\eta,x,\xi) g^{A,\lambda}_{x,\xi}(y')\ dy'd\eta\\
				&\quad = (2\pi^{-1}\lambda^{\frac{1}{2}})^{d}e^{i\varphi^{A}(y,x)}e^{i(y-x)\cdot\xi}\times \\
				&\quad \quad \quad \times \int_{\mathbb{R}^{2d}} e^{i[(y-x)-(y'-x)]\cdot(\eta-\xi)} r^{*}(y,\eta,x,\xi) e^{i\Gamma^{B}(y',x,y)} g(\lambda(y'-x))\ dy'd\eta\\
				&\quad =: (2\pi^{-1}\lambda^{\frac{1}{2}})^{d}e^{i\varphi^{A}(y,x)}e^{i(y-x)\cdot\xi} I_{\lambda,B}(r^{*},g)(y,x,\xi),
			\end{aligned}
		$$
		here $I_{\lambda,B}(r^{*},g)$ denotes the integral part. We claim $I_{\lambda,B}(r^{*},g)$ and its $\lambda y$-derivatives satisfy 
		\begin{equation}\label{eq:I_lambda B decay}
			|\partial_{\lambda y}^{\alpha}I_{\lambda,B}(r^{*},g)(y,x,\xi)| \lesssim C_{\alpha,B}\lambda^{C_{N}}\langle \lambda(y-x)\rangle^{-N},\qquad  N \in \mathbb{N}.
		\end{equation}
		 So up to some $\lambda$ factors, $I_{\lambda,B}(r^{*},g)(y,x,\xi)$ can play a similar role as the re-scaled Gaussian $g(\lambda(y-x))$, and in this sense $G_{\mathcal{R}_{2}}$ is not so different from the simpler kernel $G^{A,\lambda}$ defined in (\ref{def:G^A,lambda}).  Notice that we can re-write $I_{\lambda,B}(r^{*},g)$ as
		$$
			\begin{aligned}
				& I_{\lambda,B}(r^{*},g)(y,x,\xi)\\
				&\quad = \int_{\mathbb{R}^{d}} \bigg[e^{i\lambda(y-x)\cdot\lambda^{-1}(\eta-\xi)} r^{*}(y,\eta,x,\xi) \times\\
				& \quad \quad \times \left(\int_{\mathbb{R}^{d}} e^{-i\lambda(y'-x)\cdot\lambda^{-1}(\eta-\xi)}e^{i\Gamma^{B}(y',x,y)}g(\lambda(y'-x))\ d\lambda y'\right)\bigg]\ d\lambda^{-1}\eta,
			\end{aligned}
		$$
		 where we substituted $dy'd\eta$ with $d\lambda y' d\lambda^{-1}\eta$.  To verify \eqref{eq:I_lambda B decay}, we do $\lambda y$-differentiation term by term inside the integral $I_{\lambda,B}(r^{*},g)$. For the phase term involving $y$, we see factors that can be dominated by $\langle \lambda^{-1}(\eta-\xi)\rangle$,  since  
		\begin{equation}\label{eq:eta-xi factors}
			|\partial_{\lambda y_{j}}e^{i\lambda(y-x)\cdot\lambda^{-1}(\eta-\xi)}| = \lambda^{-1}|\eta_{j}-\xi_{j}| \leq \langle \lambda^{-1}(\eta-\xi)\rangle.
		\end{equation}
	 For the $\lambda y$-differentiation on $r^{*}$, similar to \eqref{eq:r* bounds}, by $\lambda \geq 1$ it can be bounded as 
		\begin{equation}\label{eq:r* bounds rescaled}
			|\partial_{\lambda y}^{\alpha}\partial_{\lambda^{-1}\eta}^{\beta}r^{*}(y,\eta,x,\xi)| \lesssim C_{h,\alpha,\beta}\left(\lambda^{|\beta|}+ \lambda^{|\beta|+2}\langle \lambda(y-x)\rangle^{2}\langle \lambda^{-1}(\eta-\xi)\rangle^{2}\right).
		\end{equation}
		We note the extra $\lambda^{C}$ factors on the right hand side above are due to the fact that the $\eta$-dependence of $r^{*}$ is always in the form of $(\eta-\xi)$, and each $\partial_{\lambda^{-1}\eta}$ produces a $\lambda$ factor
		$$
		\partial_{\lambda^{-1}\eta_{j}}(\eta_{j}-\xi_{j}) \leq \lambda \partial_{\lambda^{-1}\eta_{j}}\lambda^{-1}(\eta_{j}-\xi_{j}) \leq \lambda.
		$$
		For the $\lambda y$-differentiation on $e^{i\Gamma^{B}(y',x,y)}$, by the derivative bounds \eqref{eq:flux bound} of $e^{i\Gamma^{B}}$, for $\lambda \geq 1$ and $|\alpha| \geq 1$ we have  
		$$
		\begin{aligned}
			|\partial^{\alpha}_{\lambda y}e^{i\Gamma^{B}}(y',x,y)| & \leq C_{\alpha,B}(\langle \lambda(y-y')\rangle +\langle \lambda(y-x)\rangle)^{|\alpha|}\\
			& \lesssim C_{\alpha,B}(\langle \lambda(y-x)\rangle\langle \lambda(y'-x)\rangle +\langle \lambda(y-x)\rangle)^{|\alpha|}\\
			& \lesssim C_{\alpha,B}[(1+\langle \lambda(y'-x)\rangle)\langle \lambda(y-x)\rangle]^{|\alpha|}.
		\end{aligned}
		$$
		Because $g(\lambda(y'-x))$ is Schwartz in $\lambda(y'-x)$, any factor of $\langle \lambda(y'-x)\rangle$ will be harmless. To show the bound \eqref{eq:I_lambda B decay}, it remains to show that we can control factors of $\langle \lambda^{-1}(\eta-\xi)\rangle$ and $\langle \lambda(y-x)\rangle$.  Notice we can gain a decaying factor of $\langle \lambda(y-x)\rangle^{-2}$ repeatedly by integration by parts with respect to the $\lambda^{-1}\eta$-integral using
		\begin{equation}\label{eq:IBP lambda^-1 eta}
			\langle \lambda(y-x)\rangle^{-2}(1-\Delta_{\lambda^{-1} \eta})e^{i\lambda(y-x) \cdot \lambda^{-1}\eta} = e^{i\lambda(y-x) \cdot \lambda^{-1}\eta},
		\end{equation}
		and gain a decaying factor $\langle \lambda^{-1}(\eta-\xi)\rangle^{-2}$ factors by repeatedly integration by parts with respect to the $\lambda y'$-integral using
		\begin{equation}\label{eq:IBP lambda y'}
			\langle \lambda^{-1}(\eta-\xi)\rangle^{-2}(1-\Delta_{\lambda y'})e^{i\lambda^{-1}(\eta-\xi) \cdot \lambda y'} = e^{i\lambda^{-1}(\eta-\xi) \cdot \lambda y'}.
		\end{equation}
		
For the integration by parts of the $\lambda^{-1}\eta$-integral, we need to check the $\partial_{\lambda^{-1}\eta}$ derivatives of $r^{*}$ and the $\partial_{\lambda^{-1}\eta}$ derivatives of the $\lambda y'$-integral. For $r^{*}$, we use again the derivative bounds \eqref{eq:r* bounds rescaled}. For the $\partial_{\lambda^{-1}\eta}$ derivatives of the $\lambda y$-integral
		$$
		 I_{\lambda y'} := \int_{\mathbb{R}^{d}} e^{-i\lambda(y'-x)\cdot\lambda^{-1}(\eta-\xi)}e^{i\Gamma^{B}(y',x,y)}g(\lambda(y'-x))\ d\lambda y',
		$$
	each $\lambda^{-1}\eta$-differentiation only produces a factor that can be dominated by the harmless factor $\langle \lambda (y'-x)\rangle$, that is
		\begin{equation}\label{eq:y'-x factors}
			|\partial_{\lambda^{-1}\eta_{j}} e^{-i\lambda(y'-x)\cdot\lambda^{-1}(\eta-\xi)}| = \lambda|y'_{j}-x_{j}| \leq \langle \lambda (y'-x)\rangle.
		\end{equation}
		So by repeated integration by parts relying on \eqref{eq:IBP lambda^-1 eta}, we can obtain decay  in  $\langle \lambda(y-x)\rangle^{-N}$ with the cost of some $\lambda^{C_{N}}$ factor and a $\langle \lambda^{-1}(\eta-\xi)\rangle^{2}$ factor from \eqref{eq:r* bounds rescaled}. 
		 For the integration by parts \eqref{eq:IBP lambda y'} inside the $I_{\lambda y'}$ integral, we also note by the derivative bounds \eqref{eq:flux bound} of $e^{i\Gamma^{B}}$, for $\lambda \geq 1$ and $|\alpha| \geq 1$ we have  
		$$
		\begin{aligned}
			|\partial^{\alpha}_{\lambda y'}e^{i\Gamma^{B}}(y',x,y)| 
			& \lesssim C_{\alpha,B}[(1+\langle \lambda(y-x)\rangle)\langle \lambda(y'-x)\rangle]^{|\alpha|}.
		\end{aligned}
		$$
	 Combined with the fact that $g(\lambda(y'-x))$ is Schwartz in $\lambda(y'-x)$, we obtain 
		\begin{equation}\label{eq:y'-x factors II}
			\left|\partial_{\lambda y'}^{\alpha} (e^{i\Gamma^{B}(y',x,y)})g(\lambda(y'-x))\right| \lesssim C_{N,\alpha,B}(1+\langle \lambda(y-x)\rangle^{|\alpha|})\langle \lambda(y'-x)\rangle^{-N},\qquad   N \in \mathbb{N}.
		\end{equation}
		 This means by integration by parts \eqref{eq:IBP lambda y'} for the $I_{\lambda y'}$ integral, we obtain sufficient decay of $\langle \lambda^{-1}(\eta-\xi)\rangle^{-N}$ at the cost of some finite power of $\langle \lambda(y-x)\rangle$, which will be dominated by the sufficient decay of $\langle \lambda(y-x)\rangle^{-N}$ we obtained from the integration by parts \eqref{eq:IBP lambda^-1 eta} in the  $\lambda^{-1}\eta$-integral.  Then the claimed bound \eqref{eq:I_lambda B decay} holds by aforementioned integration by parts and bounds \eqref{eq:eta-xi factors}, \eqref{eq:r* bounds rescaled}, \eqref{eq:y'-x factors} and \eqref{eq:y'-x factors II}.
		
	 	With the decay property of $I_{\lambda,B}(r^{*},g)$ in \eqref{eq:I_lambda B decay}, we define $W'_{B,\lambda}$  similar to $W_{B,\lambda}$  in  \eqref{def:wavepackets product W}
		$$
		W'_{B,\lambda}(y,x^{t},\xi^{t},z^{t}) := e^{i\Gamma^{B}(y,x^{t},z^{t})} g(\lambda(y-z^{t}))I_{\lambda,B}(r^{*},g)(y,x^{t},\xi^{t}),
		$$
		so that combining \eqref{eq:flux bound} and \eqref{eq:I_lambda B decay} we obtain a decay estimate similar to \eqref{eq:W derivatives},   that is for all $N,N',N'' \in \mathbb{N}$, we have
		\begin{equation}\label{eq:W' derivatives}
			\begin{aligned}
				& \left|\partial_{\lambda y}^{\alpha}W'_{B,\lambda}(y,x^{t},\xi^{t},z^{t})\right|\\
				& \quad \lesssim C_{N,N',N'',\alpha,B}\lambda^{C_{N''}}\langle \lambda(z^{t}-x^{t})\rangle^{-N}\langle \lambda(y-z^{t})\rangle^{-N'}\langle \lambda(y-x^{t})\rangle^{-N''}.
			\end{aligned}
		\end{equation}
		 Now we can re-write $G_{\mathcal{R}_{2}}$ with $W'_{B,\lambda}$ as
		\begin{equation*}
			\begin{aligned}
				& G_{\mathcal{R}_{2}}(t;x,\xi,z,\zeta)\\
				&\quad = \int_{\mathbb{R}^{d}} \overline{g^{A,\lambda}_{z^{t},\zeta^{t}}}(y) \lambda^{\frac{d}{2}}e^{i\varphi^{A}(y,x^{t})}e^{i(y-x^{t})\cdot\xi^{t}} I_{\lambda,B}(r^{*},g)(y,x^{t},\xi^{t})\ dy\\
				&\quad = e^{i\varphi^{A}(z^{t},x^{t})}e^{i\zeta^{t}\cdot(z^{t}-x^{t})} \times\\
				&\quad \quad \times \int_{\mathbb{R}^{d}} e^{i (\xi^{t}-\zeta^{t}) \cdot (y-x^{t})} e^{i\Gamma^{B}(y,x^{t},z^{t})} g(\lambda(y-z^{t}))I_{\lambda,B}(r^{*},g)(y,x^{t},\xi^{t})\ d\lambda y\\
				&\quad = e^{i\varphi^{A}(z^{t},x^{t})}e^{i\zeta^{t}\cdot(z^{t}-x^{t})} \int_{\mathbb{R}^{d}} e^{i (\xi^{t}-\zeta^{t}) \cdot (y-x^{t})} W'_{B,\lambda}(y,x^{t},\xi^{t},z^{t})\ d\lambda y,
			\end{aligned}
		\end{equation*}
		and integration by parts again based on 
		$$
		\langle \lambda^{-1}(\zeta^{t}-\xi^{t})\rangle^{-2}(1-\Delta_{\lambda y})e^{i\lambda^{-1}(\zeta-\xi) \cdot \lambda y} = e^{i(\zeta-\xi) \cdot \lambda y}.
		$$
		Together with the decay bounds of $W'_{B,\lambda}$ in \eqref{eq:W' derivatives}, we obtain for all $N \in \mathbb{N}$ 
		$$
			|G_{\mathcal{R}_{2}}(t;x,\xi,z,\zeta)| \lesssim C_{N,B}\lambda^{C_{N}} \langle \lambda(z^{t}-x^{t})\rangle^{-N} \langle \lambda^{-1}(\zeta^{t}-\xi^{t})\rangle^{-N}.
		$$
		Combining the above bound with the $G_{\mathcal{R}_{1}}$ bound in \eqref{eq:G_1 bound},  we have the desired bound for $G_{K}$.
	\end{proof}
	
	\section{Extension to the case of time-dependent magnetic fields}
	
	In this section we consider  a  time-dependent magnetic field $B(t)$ with components $B_{jk} \in C^{1}(\mathbb{R},C^{\infty}(\mathbb{R}^{d}))$,  satisfying the assumptions in Theorem \ref{thm:time-dependent}. In particular, we recall the decay condition \eqref{eq:B(t) bounds} for $B_{jk}(t)$ and $\dot{B}_{jk}(t)$, that is, for some $\epsilon > 0$
	$$
		|\partial^{\alpha}B_{jk}(t;y)| + |\partial^{\alpha}\dot{B}_{jk}(t;y)| \leq C_{\alpha,B}\langle y\rangle^{-1-\epsilon},\qquad  |\alpha| \geq 1.
	$$
 In principal, by substituting $A(x)$ with $A(t;x)$ constructed from $B(t)$, we can adapt most of the proof for Theorem \ref{thm:main} to prove Theorem \ref{thm:time-dependent}. In the following, we only highlight the necessary changes needed beyond obvious substitution and straightforward calculations.

 Recalling the construction for the time-independent case, we define analogously the time-dependent vector potentials $A(t)$ and phase function $\varphi^{A(t)}$ by 
	\begin{align}\label{def:phi^A(t)} \nonumber 
			A_{j}(t;y,x)  & := -\sum_{k=1}^{d}\int_{0}^{1} s(y_{k}-x_{k}) B_{jk}(t;x+s(y-x))\ ds,\quad y \in \mathbb{R}^{d},\\
			A_j(t;y) & :=-\sum_{k=1}^d \int_0^1 sy_k B_{j k}(t;sy)\  d s = A_{j}(t;y,0),\\ \nonumber
			\varphi^{A(t)}(y,x) &  := (y-x)\cdot \int_{0}^{1} A(t;(1-s)x+sy,0)\ ds.
	\end{align}
	We can further construct time-dependent magnetic pseudo-differential operators
	$$
	\operatorname{Op}^{A(t)}(h)u(y) := (2\pi)^{-d}\int_{\mathbb{R}^{2d}} e^{i(y-y')\cdot \eta} e^{-i\varphi^{A(t)}(y',y)} h(\tfrac{y+y'}{2},\eta)u(y')\ dy'd\eta, \quad y \in \mathbb{R}^{d},
	$$
	and time-dependent magnetic wavepackets
	$$
	g^{A(t),\lambda}_{x,\xi}(y) := \lambda^{\frac{d}{2}}e^{i\xi\cdot (y-x)}g(\lambda(y-x))e^{i\varphi^{A(t)}(y,x)}, \quad y \in \mathbb{R}^{d},
	$$
	and also accordingly the transform $\mathcal{T}^{A(t)}_{\lambda}$ and the  magnetic modulation  spaces $M^{m,p}_{A(t)}$.
	
	With the above constructions, Lemma \ref{lem:flat approximation} holds with obvious changes since the time dependence does not play a role there. It suffices to substitute the $A$ and $B$ in Lemma \ref{lem:flat approximation} with the time-dependent ones $A(t)$ and $B(t)$, in particular \eqref{def:m_h}, \eqref{def:R_1} and \eqref{def:R_2} should be replaced by 
	\begin{equation} \label{def:m_h R_1 R_2 with t}
		\begin{aligned}
			m^{A(t)}(h)(x,\xi) & := h(x,\xi) - \partial_{\eta}h(x,\xi)\cdot (\xi+A(t;x))\\
			\mathcal{R}_{1}(t)g^{A(t),\lambda}_{x,\xi}(y) & :=  \partial_{\eta}h(x,\xi)\cdot\left( - r^{(1)}_{x}(A(t;y,x)) + r^{(1)}_{x}(A(t;x,y))\right)g^{A(t),\lambda}_{x,\xi}(y)\\
			\mathcal{R}_{2}(t)g^{A(t),\lambda}_{x,\xi}(y) & := \operatorname{Op}_{KN}^{A(t)}\left(r^{(1)}_{x,\xi}(h) + h_{r}\right)g^{A(t),\lambda}_{x,\xi}(y).
		\end{aligned}
	\end{equation}

	On the other hand, some further modification is required for Lemma \ref{lem:wavepacket solution} and Proposition \ref{prop:parametrix}, because it involves $t$-derivatives. We define a new vector field $\tilde{H}^{B(t)}$
	$$
	\tilde{H}^{B(t)}_{j} := \left(\partial_{\xi_{j}}h(t)\partial_{x_{j}},-\partial_{x_{j}}h(t)\partial_{\xi_{j}}+\dot{A}_{j}(t;x)\partial_{\xi_{j}}+\sum_{k=1}^{d}B_{kj}(t;x)\partial_{\xi_{k}}h(t)\partial_{\xi_{j}}\right),\quad 1\leq j \leq d,
	$$
	and the associated flow $\chi'_{B}(h)(t,s): (x^{s},\xi^{s}) \mapsto (x^{t},\xi^{t})$ satisfying
	$$
	\begin{cases}
		\dot{x^{t}_{j}} = \partial_{\xi_{j}}h(t;x^{t},\xi^{t}),\\
		\dot{\xi^{t}_{j}} =  -\partial_{x_{j}}h(t;x^{t},\xi^{t}) + \dot{A}_{j}(t;x^{t}) + \sum_{k=1}^{d}B_{kj}(t;x^{t})\partial_{\xi_{k}}h(t;x^{t},\xi^{t}).
	\end{cases}
	$$
	Notice that $\tilde{H}^{B(t)}$ is also divergence-free and by our assumptions of $B(t)$ and construction of $A(t)$,  we have
	$$
	|\dot{A}_{j}(t;x^{t})| \lesssim 1 + |x^{t}|,
	$$
	so it is straightforward to verify that Lemma \ref{lem:flow} and Lemma \ref{lem:t-averaging} also hold for $\chi'_{B}(h)$.
	
	In order to adapt Proposition \ref{prop:parametrix} for the time-dependent magnetic potentials, we define
	\begin{equation}\label{def:g^A(t)}
		\begin{gathered}
			\psi'(t;x,\xi) := \int_{0}^{t} -m^{A(\tau)}(h)(\tau;x^{\tau},\xi^{\tau})\ d\tau,\\
			g^{A(t),\lambda}_{\psi'}(t;y,x,\xi) := e^{i\psi'(t)}\chi'_{B}(h)(t,0)g^{A(t),\lambda}_{x,\xi}(y).
		\end{gathered}
	\end{equation}
	
	\begin{lem}\label{lem:wavepacket solution with A(t)}
		The wavepacket $g^{A(t),\lambda}_{\psi'}$ associated to the flow $\chi'_{B}(h)$ satisfies
		$$
		(D_{t} + \operatorname{Op}^{A(t)}(h))g^{A(t),\lambda}_{\psi'}(t;y,x,\xi) = (\mathcal{R}_{1} + \mathcal{R}_{2} +  \mathcal{R}_{3})(t)g^{A(t),\lambda}_{\psi'}(t;y,x,\xi),
		$$
		where $\mathcal{R}_{1}$ and $\mathcal{R}_{2}$ are as defined in \eqref{def:m_h R_1 R_2 with t}, and  for some $c \in (0,1)$,  
		\begin{equation}\label{def:R_3}
			\mathcal{R}_{3}(t;y,x^{t}) := \sum_{j=1}^{d} (y_{j}-x^{t}_{j}) \int_{0}^{1} \partial_{x}\dot{A}_{j}(t;x^{t}+cs(y-x^{t})) \cdot s(y-x^{t})\ ds,
		\end{equation}
		 which satisfies   the bound
		\begin{equation}\label{eq:R_3 bound}
			|\mathcal{R}_{3}(t;y,x^{t})| \leq C_{B}\langle y-x^{t}\rangle^{2},\qquad  t \in \mathbb{R}.
		\end{equation}
	\end{lem}
	
	\begin{proof}
		By the same calculation as in Lemma \ref{lem:flat approximation}, we have
		$$
		\begin{aligned}
			\operatorname{Op}^{A(t)}(h)g^{A(t),\lambda}_{\psi'}(t;y,x,\xi) & = i\sum_{j=1}^{d} \left(\partial_{\xi_{j}}h\partial_{x_{j}}-(\partial_{x_{j}}h - \sum_{k=1}^{d} B_{kj}(t;x)\partial_{\xi_{k}}h)\partial_{\xi_{j}}\right)g^{A(t),\lambda}_{\psi'}(t;y,x,\xi)\\
			& \quad \quad + (m^{A(t)}(h)+\mathcal{R}_{1}(t)+\mathcal{R}_{2}(t))g^{A(t),\lambda}_{\psi'}(t;y,x,\xi)
		\end{aligned}
		$$
		with the terms defined in \eqref{def:m_h R_1 R_2 with t}. Then it suffices to show  that the $t$-derivative of $g^{A(t),\lambda}_{\psi'}$ is of the form 
		\begin{equation}\label{eq:t-derivative g^A(t)}
			\begin{aligned}
				D_{t}g^{A(t),\lambda}_{\psi'}(t;y,x,\xi) & = -i\sum_{j=1}^{d}\left(\tilde{H}^{B(t)}_{j}-\dot{A}_{j}(t;x)\partial_{\xi_{j}}\right)g^{A(t),\lambda}_{\psi'}(t;y,x,\xi)\\
				& \quad \quad -\left(m^{A(t)}(h) - \mathcal{R}_{3}(t)\right)g^{A(t),\lambda}_{\psi'}(t;y,x,\xi),
			\end{aligned}
		\end{equation}
		with $\mathcal{R}_{3}$ satisfying the bound \eqref{eq:R_3 bound}. We claim
		$$
		D_{t}e^{i\varphi^{A(t)}(y,x)} = \left((y-x)\cdot \dot{A}(t;x) + \mathcal{R}_{3}(t;y,x)\right)e^{i\varphi^{A(t)}(y,x)},
		$$
		which combined with the construction of $g^{A(t),\lambda}_{\psi'}$ as in \eqref{def:g^A(t)} gives \eqref{eq:t-derivative g^A(t)}. Firstly notice
		$$
		\partial_{t}e^{i\varphi^{A(t)}(y,x)} = i\dot{\varphi^{A(t)}}(y,x)e^{i\varphi^{A(t)}(y,x)}, \quad \dot{\varphi^{A(t)}} := d\varphi^{A(t)}/dt.
		$$
		Recall our construction in \eqref{def:phi^A(t)}, from which we get 
		$$
		\begin{aligned}
			\dot{\varphi^{A(t)}}(y,x) & = (y-x)\cdot \int_{0}^{1} \dot{A}(t;(1-s)x+sy,0)\ ds\\
			& = \sum_{j=1}^{d} (y_{j}-x_{j}) \int_{0}^{1} \dot{A}_{j}(t;x+s(y-x),0)\ ds\\
			& = \sum_{j=1}^{d} (y_{j}-x_{j}) \int_{0}^{1} \dot{A}_{j}(t;x)\ ds\\
			& \quad \quad + \sum_{j=1}^{d} (y_{j}-x_{j}) \int_{0}^{1} \partial_{x}\dot{A}_{j}(t;x+cs(y-x)) \cdot s(y-x)\ ds,
		\end{aligned}
		$$
		where we used the Taylor expansion of $\dot{A}_{j}$ at $x$. With $\mathcal{R}_{3}$ as defined in \eqref{def:R_3}, we see that 
		$$
		\dot{\varphi^{A(t)}}(y,x) = (y-x)\cdot \dot{A}(t;x) + \mathcal{R}_{3}(t;y,x).
		$$
		For the first term on the right hand side above we use the identity
		$$
		[(y-x)\cdot \dot{A}(t;x)]g^{A(t),\lambda}_{\psi'}(t;y,x,\xi) = -i\dot{A}(t;x) \cdot \partial_{\xi}g^{A(t),\lambda}_{\psi'}(t;y,x,\xi).
		$$
		For the $\mathcal{R}_{3}$ term 
		$$
		\mathcal{R}_{3}(t;y,x) := \sum_{j=1}^{d} (y_{j}-x_{j}) \int_{0}^{1} \partial_{x}\dot{A}_{j}(t;x+cs(y-x)) \cdot s(y-x)\ ds,
		$$
		we see explicitly for the integrand 
		$$
		\begin{aligned}
			\partial_{i}\dot{A}_{j}(t;x+cs(y-x)) & = -\partial_{i}\left(\sum_{k=1}^d \int_0^1  \theta (x_{k}+cs(y_{k}-x_{k}))   \dot{B}_{j k}(t;\theta (x+cs(y-x)))\  d \theta\right)\\
			& = - \int_0^1  (\theta-cs)   \dot{B}_{j i}(t;\theta (x+cs(y-x)))\  d \theta\\
			& \quad \quad -\sum_{k=1}^d \int_0^1 \theta (x_{k}+cs(y_{k}-x_{k})) \partial_{i}\dot{B}_{j k}(t;\theta (x+cs(y-x)))\  d \theta.
		\end{aligned}
		$$
		By the boundedness of  $B(t), \dot{B}(t)$  and the decay assumption \eqref{eq:B(t) bounds}, we see that 
		$$
		\sup_{x,y} |\partial_{i}\dot{A}_{j}(t;x+cs(y-x))| \leq C_{B}, \qquad  t \in \mathbb{R}.
		$$
		By the above bound and our definition \eqref{def:R_3} of $\mathcal{R}_{3}$, we have the bound \eqref{eq:R_3 bound} as desired.
	\end{proof}
	
	With Lemma \ref{lem:wavepacket solution with A(t)}, we can now adapt Proposition \ref{prop:parametrix} to the case of $A(t)$.
	
	\begin{prop}
		Let $0\leq s,t \leq T$. We define the operator $\tilde{S}(t,s)$ by 
		\begin{equation}
			\tilde{S}(t,s)u_{0}(y) := \int_{\mathbb{R}^{2d}\times\mathbb{R}^{d}} g^{A(t),\lambda}_{\psi'}(t;y,x,\xi) \overline{g^{A(s),\lambda}_{\psi'}}(s;y',x,\xi)u_{0}(y')\ dy'dxd\xi,
		\end{equation}
		for $u_{0} \in \mathcal{S}(\mathbb{R}^{d})$. We also define the operator $K(t,s)$ by
		\begin{equation}
			K(t,s)u_{0}(y) := \int (\mathcal{R}_{1}+\mathcal{R}_{2}+\mathcal{R}_{3})(t)g^{A(t),\lambda}_{\psi}(t;y,x,\xi) \overline{g^{A(s),\lambda}_{\psi}}(s;y',x,\xi)u_{0}(y')\ dy'dxd\xi,
		\end{equation}
		for $u_{0} \in \mathcal{S}(\mathbb{R}^{d})$. We then have
		\begin{equation}
			(D_{t}+\operatorname{Op}^{A(t)}(h))\tilde{S}(t,s)u_{0} = K(t,s)u_{0},\quad u_{0} \in \mathcal{S}(\mathbb{R}^{d}).
		\end{equation}
		Moreover, for  $1 \leq p \leq \infty, m \in \mathbb{R}$ and fixed $t,s$,  the operator $\tilde{S}(t,s)$ is bounded from $M^{m,p}_{A(s)}(\mathbb{R}^{d})$ to $M^{m,p}_{A(t)}(\mathbb{R}^{d})$.
	\end{prop}
	
 With the above proposition, we can prove adapted versions of Lemma \ref{lem:G bound} and Lemma \ref{lem:volterra} by making the obvious changes of substituting $A$ with $A(t)$ to the relevant proofs. The adapted arguments hold because the time-dependent $\mathcal{R}_{1}(t),\mathcal{R}_{2}(t),\mathcal{R}_{3}(t)$ have bounds that are comparable to the time-independent case. Then eventually Theorem \ref{thm:time-dependent} can be derived from the adapted lemmas similarly to the derivation of Theorem \ref{thm:main}.

\bibliographystyle{plain}
\bibliography{magnetic.bib}
	
\end{document}